\documentclass[10pt]{article}
\usepackage{amssymb,color}
\usepackage{amsmath,graphicx,enumerate,textcomp}
\usepackage[abbrev]{amsrefs}
\usepackage[a4paper,top=1.in, bottom=1.5in, left=.8in, right=.8in, headsep=0.in, centering]{geometry}
\usepackage{amssymb,amsfonts,amsthm}
\usepackage{epsfig}
\numberwithin{equation}{section}
\usepackage{lipsum}
\usepackage{graphicx}
\usepackage{esint}

\newcommand{\eps}{\varepsilon}

\newtheorem{theorem}{Theorem}[section]
\newtheorem{lemma}{Lemma}[section]
\newtheorem{corollary}{Corollary}[section]

\newtheorem{proposition}{Proposition}[section]

\newtheorem{conjecture}{Conjecture}[section]

\title{Uniqueness of traveling fronts in premixed flames with stepwise ignition-temperature kinetics and fractional reaction order.   }

\author{ Amanda Matson
\thanks{Department of Mathematical Sciences, 
Kent State University,
 Kent, OH 44242, USA. E-mail: {\tt amatson2@kent.edu }}
\and  Claude-Michel Brauner
\thanks{Institut de Math\'ematiques de Bordeaux UMR CNRS 5251, Universit\'e de Bordeaux, 33405 Talence, France.  E-mail:
{\tt claude-michel.brauner@u-bordeaux.fr}}
\and Peter  V. Gordon
\thanks{Department of Mathematical Sciences, 
Kent State University,
 Kent, OH 44242, USA. E-mail: {\tt gordon@math.kent.edu}}
 }

\begin{document}
\maketitle

\begin{abstract}
In this paper, we consider a reaction-diffusion system describing the propagation of flames under the assumption of ignition-temperature kinetics and fractional reaction order.  
It was shown in \cite{CMB21} that this system admits a traveling front  solution. In the present work,  we show that this  traveling front  is unique up to translations.
We also study some qualitative properties of this solution using the combination of  formal asymptotics and numerics.  Our findings allow conjecture that the velocity of the propagation
of the flame front is a decreasing function of all of the parameters of the problem: ignition temperature, reaction order and an inverse of the Lewis number.
\end{abstract}

\medskip


\bigskip


\section{Introduction}

The canonical constant density approximation model of flame propagation in one dimensional formulation  reads \cite{BNS85,ZBLM,Will,Law}:
\begin{eqnarray}\label{eq:i1}
\left\{ 
\begin{array}{l}
T_{\bar t}=T_{\bar x\bar x}+\Omega(T,C), \\
C_{\bar t}=\Lambda C_{\bar x \bar x}-\Omega(T,C),
\end{array}
\right.
\end{eqnarray}
where $T$ and $C$ are appropriately normalized temperature and concentration of the deficient reactant, $\bar x\in \mathbb{R},$~ $\bar{t}>0$ are normalized  spatiotemporal coordinates, $\Lambda>0$ is an inverse of  the Lewis number and $\Omega$ is the reaction rate.
The reaction rate is typically specified as: 
\begin{eqnarray}\label{eq:i2}
\Omega(T,C):=
\left\{
\begin{array}{llll}
C^{\alpha} F(T) & T\ge \theta & \mbox{and}&  C>0,\\
0 & T<\theta  & \mbox{and/or} & C=0,
\end{array}
\right.
\end{eqnarray}
where $\alpha\ge0$ is the reaction order, $\theta \in (0,1)$ is the ignition temperature  and $F(T)$ is a positive non-decreasing function that characterizes the enhancement of chemical reaction with the increase 
of the temperature.

 The studies of system \eqref{eq:i1} trace back to pioneering  works of Frank-Kamenetskii,  Semenov and Zeldovich  in 1930's and 1940's  \cite{ZBLM,FK,Sem}. This system was then  analyzed by mathematicians and physicists alike. 
 The substantial body of   literature  concerning  system \eqref{eq:i1} is dedicated to the analysis of traveling front solutions,  that is special solutions of the form:
 \begin{eqnarray}\label{eq:i3}
 T(\bar x, \bar t):= u(\xi), \quad C(\bar x,\bar t):= v(\xi), \quad \xi=\bar x+c \bar t,
 \end{eqnarray} 
 where $c>0$ is an a-priori unknown speed of propagation.
 In a context of combustion, such solutions represent flame fronts  propagating  with a constant speed from burned state far to the right to unburned state far to the left.

When $\alpha\ge 1$  system \eqref{eq:i1}, after substitution of an ansatz \eqref{eq:i3}, reduces to the following system of ODE's on a real line:
\begin{eqnarray}\label{eq:i4}
u_{\xi\xi}-cu_{\xi}+\Omega(u,v)=0, \quad \Lambda v_{\xi\xi}-cv_{\xi}-\Omega(u,v)=0, \quad \xi\in \mathbb{R},
\end{eqnarray}
complemented with boundary like conditions:
\begin{eqnarray}\label{eq:i5}
u \to 1, \quad v\to 0 \quad \mbox{as} \quad \xi \to \infty,\quad \mbox{and} \quad
u \to 0, \quad v\to 1 \quad \mbox{as} \quad \xi \to -\infty.
\end{eqnarray}
 Conditions \eqref{eq:i5} prescribe the steady temperature and reactant concentration far ahead ($\xi\to- \infty$) and far behind  ($\xi \to \infty$) the flame front.  

Since solutions of  \eqref{eq:i4}, \eqref{eq:i5} are translationally invariant, we fix translations by imposing a constraint:
\begin{eqnarray}\label{eq:i6}
u(\xi_{ign})=\theta,
\end{eqnarray}
where $\xi_{ign}$ is an arbitrary  fixed number.

The constraint \eqref{eq:i6}  fixes the position of an
 {\it ignition interface},  the unique  position where the temperature is equal to the ignition temperature $\theta$.  Hence, when crossing the ignition interface, the reaction rate jumps from some positive value to zero while preserving continuity of the temperature and concentration of reactant as well as
 their fluxes. 
  Consequently, the mixture ahead of the ignition interface ($\xi<\xi_{ign}$) is in a non-reactive state,  whereas at and behind the ignition interface  ($\xi\ge \xi_{ign}$),
 the chemical reaction takes place.  
  We note that  uniqueness of  the ignition  interface follows from  the monotonicity of any solution of problem \eqref{eq:i4},\eqref{eq:i5} that results from the fact that $\Omega(u,v)\ge0$ and can be directly verified.
 
 In view of the discussion above, system \eqref{eq:i4} is equivalent  to the following one: 
 \begin{eqnarray}\label{eq:i7}
 \left\{
\begin{array}{lll}
u_{\xi\xi}-cu_{\xi}=0, &  \Lambda v_{\xi\xi}-cv_{\xi}=0, & \xi\in (-\infty,\xi_{ign}),\\
u_{\xi\xi}-cu_{\xi}+v^{\alpha}F(u)=0, & \Lambda v_{\xi\xi}-cv_{\xi}-v^{\alpha}F(u)=0, & \xi\in (\xi_{ign} ,\infty),
\end{array}
\right.
\end{eqnarray}
complemented with conditions of continuity of a solution and its first derivatives when crossing the interface: 
\begin{eqnarray}\label{eq:i7a}
[u]=[v]=[u_{\xi}]=[v_{\xi}]=0 \quad \mbox{at} \quad \xi=\xi_{ign}.
\end{eqnarray}
Here and below $[\cdot]$ stands for a jump of the quantity. A sketch of a traveling front solution for system \eqref{eq:i1} with the reaction order $\alpha\ge 1$ is depicted on Figure \ref{f:1}.
\begin{figure}[h]
\centering \includegraphics[width=4.in]{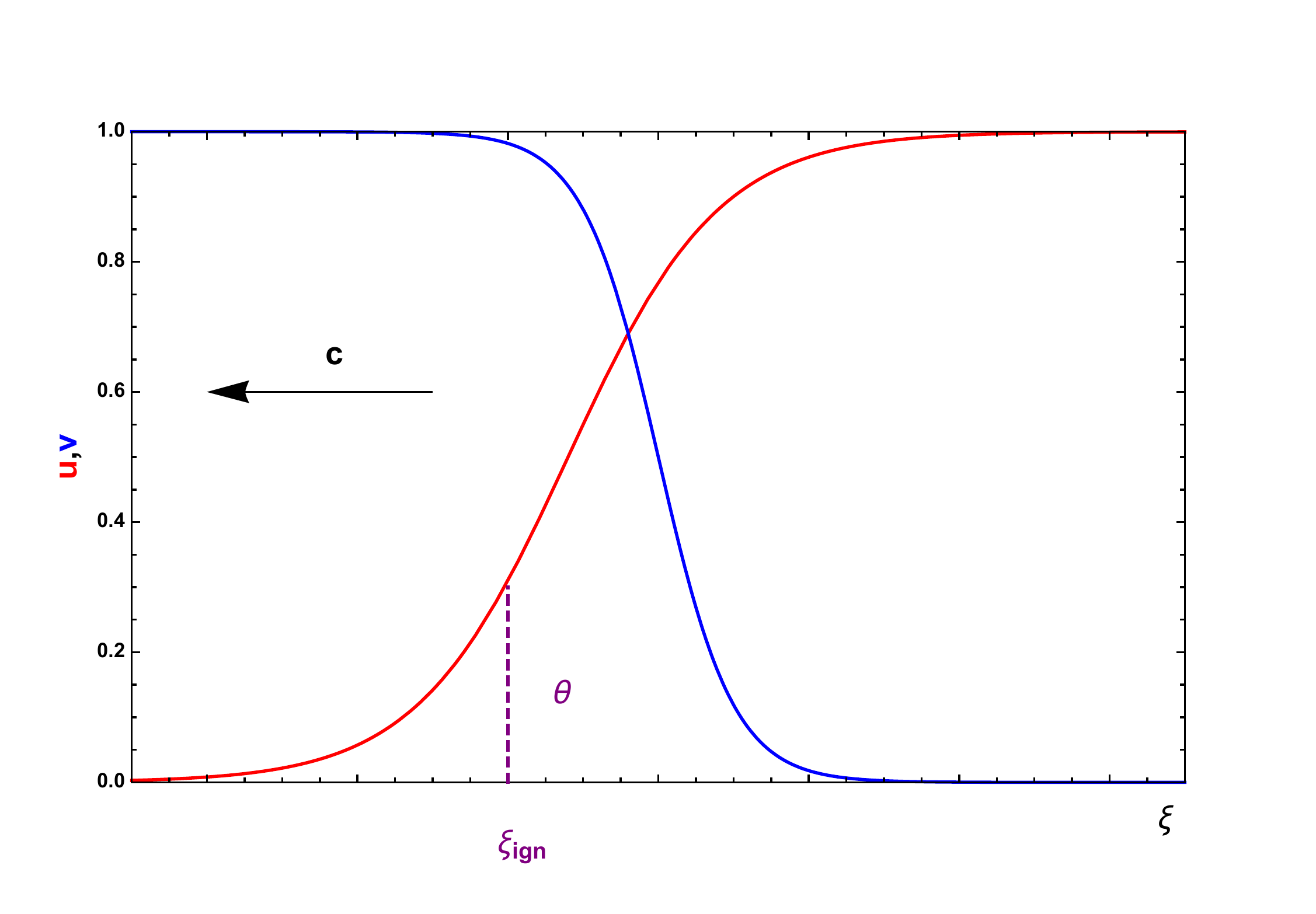}
\caption{Sketch of a traveling front solution for system \eqref{eq:i1} with $\alpha\ge 1$. Flame (red) and deficient  reactant (blue) fronts propagating from burned ($(u,v)=(1,0)$) to  unburned ($(u,v)=(0,1)$)
 states at $\xi\to \infty$ and $\xi\to -\infty$ respectively
with the constant speed $c>0$. The arrow indicates  the direction of propagation.The position of an ignition interface where the temperature is equal to an ignition one ($u=\theta$) is indicated by $\xi_{ign}$. }
\label{f:1} 
\end{figure}

System \eqref{eq:i4}, \eqref{eq:i5}, \eqref{eq:i6} (equivalently system \eqref{eq:i5}, \eqref{eq:i6}, \eqref{eq:i7}, \eqref{eq:i7a}) was extensively studied in the literature.  In the special case when $\Lambda=1,$ this system  reduces to a single equation. 
This equation  is well understood by now, details can be found in many books and review articles on the subject, for example  \cite{ ZBLM, Xin_rev, GK_book,V3 }.
In particular, it is well known that in this case \eqref{eq:i4}, \eqref{eq:i5}, \eqref{eq:i6}  admits a unique traveling front solution for any $\Lambda\in (0,\infty)$ fixed.
The general case when  $\Lambda\ne 1$ is substantially more complex, and the complete  understanding of this case is still lacking.   Review of many relevant results concerning this general case can be found in \cite{V3,Volpert}.
In particular, it was shown in \cite{BNS85} that when the reaction rate $\alpha$ is a positive  integer  then  system \eqref{eq:i4}, \eqref{eq:i5}, \eqref{eq:i6}  admits
a traveling front solution. Moreover, when $\alpha=1$ and  $\Lambda \in (0,1]$
this solution is unique  \cite{Kan63}. The question whether uniqueness of a solution holds for $\alpha=1$ and $\Lambda>1$ is still open.

When $0<\alpha<1,$  non-linear term \eqref{eq:i2} becomes non-Lipschitz at $C=0.$ Consequently,  the system that describes traveling fronts for \eqref{eq:i1} changes substantially. After substitution of an ansatz \eqref{eq:i3} into \eqref{eq:i1} this system reads:
\begin{eqnarray}\label{eq:i7b}
\left\{
\begin{array}{lll}
u_{\xi\xi}-cu_{\xi}=0, & \Lambda v_{\xi\xi}-cv_{\xi}=0, & \xi\in (-\infty,\xi_{ign}),\\
u_{\xi\xi}-cu_{\xi}+v^{\alpha}F(u)=0, & \Lambda v_{\xi\xi}-cv_{\xi}-v^{\alpha}F(u)=0, & \xi\in (\xi_{ign},\xi_{tr}),\\
u_{\xi\xi}-cu_{\xi}=0, & v=0,& \xi\in (\xi_{tr},  \infty),
\end{array}
\right.
\end{eqnarray}
where $\xi_{tr}>\xi_{ign}$ is an a-priori unknown constant which we refer to as a position of  {\it trailing interface}. The trailing interface, in the context of combustion, indicates  the leftmost point where 
entire reactant available for the reaction is  fully consumed.  As a result of it, as can be easily checked, the temperature at the trailing interface as well as at all points to the right  of it is equal to
the temperature of the fully reacted mixture ($u=1, \xi\ge \xi_{tr}$).   Hence,   the distance between the ignition and trailing interfaces $R=\xi_{tr}-\xi_{ign}$ is the length of the reaction zone.

In view of the discussion above, the last equation in \eqref{eq:i7b} can be solved explicitly and the solution reads:
\begin{eqnarray}\label{eq:i8}
u=1, & v=0,& \xi\in (\xi_{tr},\infty).
\end{eqnarray}
Hence,  system \eqref{eq:i7b} reduces to
\begin{eqnarray}\label{eq:i9}
\left\{
\begin{array}{lll}
u_{\xi\xi}-cu_{\xi}=0, & \Lambda v_{\xi\xi}-cv_{\xi}=0, & \xi\in (-\infty,\xi_{ign}),\\
u_{\xi\xi}-cu_{\xi}+v^{\alpha}F(u)=0, & \Lambda v_{\xi\xi}-cv_{\xi}-v^{\alpha} F(u)=0, & \xi\in (\xi_{ign},\xi_{tr}),
\end{array}
\right.
\end{eqnarray}
complemented with boundary conditions at the trailing interface:
\begin{eqnarray}\label{eq:i10}
u=1, \quad u_{\xi}=0, \quad v=v_{\xi}=0 \quad \mbox{at} \quad \xi=\xi_{tr},
\end{eqnarray}
that manifest continuity of the temperate and reactant and their fluxes when crossing the trailing  interface.
Problem \eqref{eq:i9}, \eqref{eq:i10} should be considered with constraint \eqref{eq:i6}, continuity conditions at the ignition interface \eqref{eq:i7a} and boundary like conditions
far ahead of the ignition interface:
\begin{eqnarray}\label{eq:i11}
u \to 0, \quad v\to 1 \quad \mbox{as} \quad \xi \to- \infty.
\end{eqnarray} 
We note that the main principle difference of a system describing traveling front solution for system \eqref{eq:i1} with $\alpha\ge1$ and $0<\alpha<1$ is that the latter involves a trailing interface which position is a-priori unknown.
We also note that any solution of  \eqref{eq:i6}, \eqref{eq:i7a},  
 \eqref{eq:i9},\eqref{eq:i10}, \eqref{eq:i11} is monotone so for any such solution positions of both ignition and trailing interfaces are uniquely defined. The sketch of the traveling front solution for this system is depicted in Figure \ref{f:2}.
\begin{figure}[h!]
\centering \includegraphics[width=4.in]{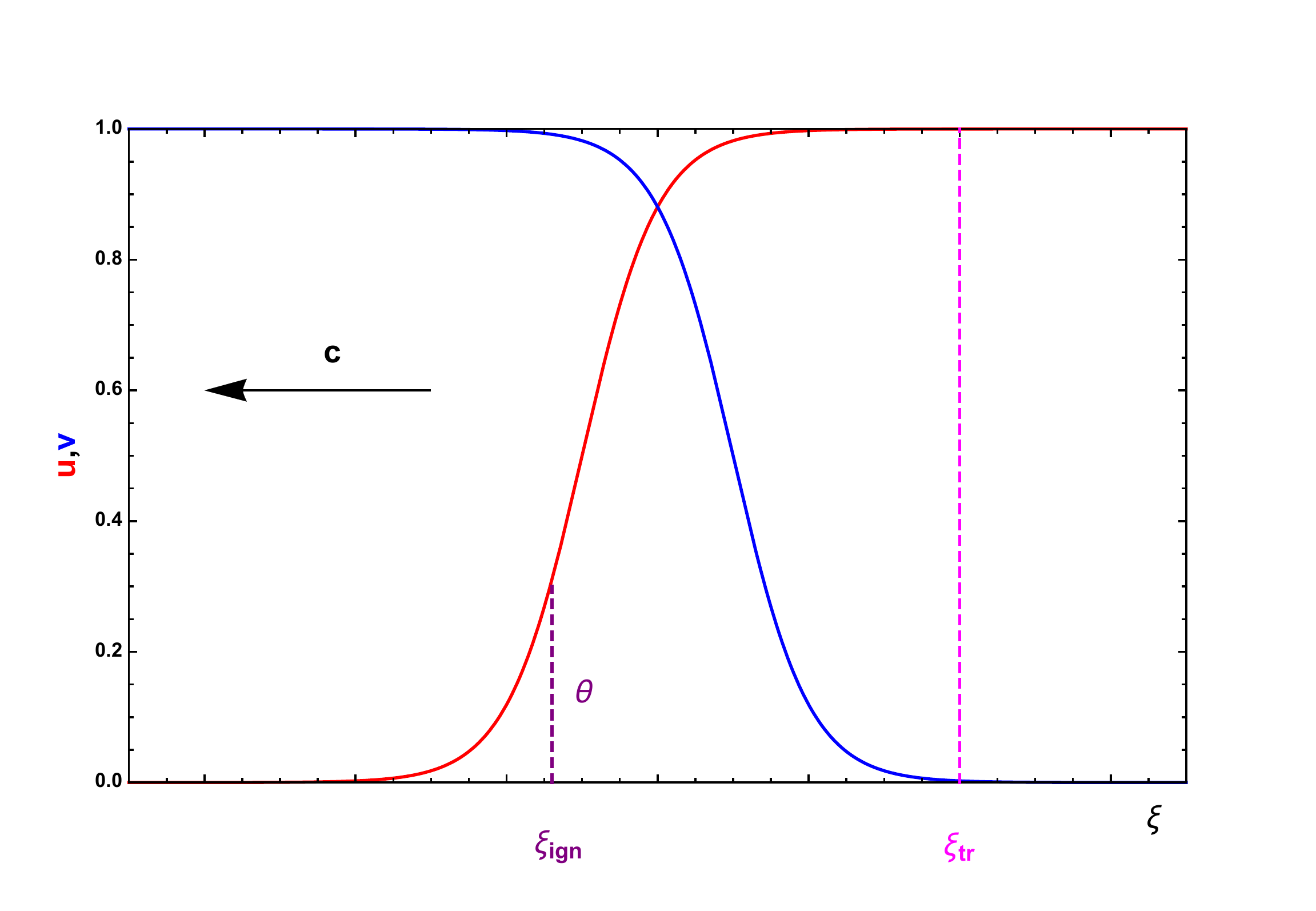}
\caption{Sketch of a traveling front solution for system \eqref{eq:i1} with $0<\alpha< 1$. Flame (red) and deficient  reactant (blue) fronts propagating from burned ($(u,v)=(1,0)$) to an unburned ($(u,v)=(0,1)$)
 states at $\xi\to \infty$ and $\xi\to -\infty$ respectively
with the constant speed $c>0$. The arrow indicates  the direction of propagation.The position of an ignition interface where the temperature is equal to an ignition one ($u=\theta$) is indicated by $\xi_{ign}.$
The position of the trailing interface, the point to the right of which temperature and concentration of deficient  are identically one and zero ($u=1,v=0$), respectively is indicated by  $\xi_{tr}.$
   }
\label{f:2} 
\end{figure}

Despite the fact that reactions  of fractional order ($0<\alpha<1$) are common in various combustion applications,  for example in  high pressure  combustion systems (see e.g \cite{Law}), 
problem \eqref{eq:i6}, \eqref{eq:i7a},  
 \eqref{eq:i9},\eqref{eq:i10}, \eqref{eq:i11} has not received appropriate attention in mathematical literature partially due to the fact that the non-Lipschitz reaction rate creates multiple technical difficulties for the analysis.
To the best of our knowledge, a recent  paper \cite{CMB21} is the only mathematical paper  where this system  was analyzed.  
In \cite{CMB21} the authors considered \eqref{eq:i6}, \eqref{eq:i7a},   \eqref{eq:i9}, \eqref{eq:i10}, \eqref{eq:i11} with a stepwise ignition-temperature kinetics, that is with  $F(T)$ being a Heaviside step-function
\begin{eqnarray}\label{eq:i12}
F(T):=
\left\{
\begin{array}{llll}
1  & T\ge \theta, \\
0 & T<\theta.  
\end{array}
\right.
\end{eqnarray}
The particular choice of stepwise ignition-temperature kinetics, from perspective of applications, is justified by the fact that such kinetic is apparently the most appropriate approximation of overall reaction kinetics
for certain hydrogen-oxygen/air and ethylene-oxygen mixtures as evident from multiple theoretical and numerical studies of detailed  reaction mechanism in such mixtures see \cite{SW} and references therein.

Under the assumption of  stepwise ignition-temperature kinetics,  system \eqref{eq:i9} simplifies to :
\begin{eqnarray}\label{eq:i13}
\left\{
\begin{array}{lll}
u_{\xi\xi}-cu_{\xi}=0, & \Lambda v_{\xi\xi}-cv_{\xi}=0, & \xi\in (-\infty,\xi_{ign}),\\
u_{\xi\xi}-cu_{\xi}+v^{\alpha}=0, & \Lambda v_{\xi\xi}-cv_{\xi}-v^{\alpha}=0, & \xi\in (\xi_{ign},\xi_{tr}).
\end{array}
\right.
\end{eqnarray}

In \cite{CMB21} it was shown that problem \eqref{eq:i6}, \eqref{eq:i7a},  \eqref{eq:i13}, \eqref{eq:i10}, \eqref{eq:i11}   admits a solution, however, the question of multiplicity of solution was not addressed there.
In this paper, we prove that a solution constructed in  \cite{CMB21} is unique and discuss some qualitative properties of the solution.

 The paper is organized as follows. In section 2, we set up the stage by giving some preliminaries and  stating the  main result. 
 In section 3, we give a proof of necessary lemmas and propositions used in the proof of the main result.
 In the last section, we discuss certain properties of traveling front solutions and present results of numerical simulations. In particular, we give some formal arguments 
 based on asymptotic and numerical studies of the problem that the velocity of propagation of the flame front  is a decreasing function of  the ignition temperature $\theta$, inverse of the Lewis number $\Lambda$ and
 the reaction order $\alpha$.

\section{Preliminaries and the statement of the main result.}

In this section, we give an alternative formulation of problem \eqref{eq:i6}, \eqref{eq:i7a},  \eqref{eq:i13}, \eqref{eq:i10}, \eqref{eq:i11},  and state the main result of this paper.
Let us first note that in a view of translational invariants of traveling fronts, we can fix the position  of trailing interface $\xi_{tr}$ and treat $\xi_{ign}$ as a-priori unknown.
Hence, we set
$\xi_{tr}=0$ and $\xi_{ign}=-R$ where $R>0$ is a-priori unknown. With this alteration, the problem describing traveling front solutions for \eqref{eq:i1} with $\alpha\in(0,1)$ and step-wise reaction kinetics
discussed in the introduction  reads: 
\begin{eqnarray}\label{eq:1}
\left\{
\begin{array}{llll}
u_{\xi\xi}-cu_{\xi}=0, & \Lambda v_{\xi\xi}-c v_{\xi} =0, & v>0, & \xi\in(-\infty,-R),\\
u_{\xi\xi}-cu_{\xi}=-v^{\alpha}, & \Lambda v_{\xi\xi}-c v_{\xi} =v^{\alpha}, & v>0 ,& \xi\in(-R,0),
\end{array}
\right.
\end{eqnarray}
This system is complemented  with boundary like conditions far ahead of the combustion front: 
\begin{eqnarray}\label{eq:4}
 u\to 0 \quad v\to 1 \quad \quad \mbox{as} \quad  \xi \to-\infty, 
 \end{eqnarray}
continuity  conditions for a solution and its first derivatives at the ignition interface:
\begin{eqnarray}\label{eq:2}
[u]=[v]=[u_{\xi}]=[v_{\xi}] =0\quad \mbox{at} \quad \xi=-R, 
\end{eqnarray}
boundary condition at the ignition interface:
\begin{eqnarray}\label{eq:3}
u=\theta, \quad \mbox{at} \quad \xi=-R, 
\end{eqnarray}
and boundary conditions at the trailing interface
\begin{eqnarray}\label{eq:5}
u=1, \quad u_{\xi}=0, \quad  v=v_{\xi}=0  \quad \mbox{at} \quad \xi=0
\end{eqnarray}
Straightforward computations show that for any solution  of problem \eqref{eq:1} satisfying \eqref{eq:2} and \eqref{eq:3} has:
\begin{eqnarray}\label{eq:6}
u(\xi)=\theta \exp(c(\xi+R)), \quad  \xi\le -R,
\end{eqnarray}
and thus
\begin{eqnarray}
u=\theta, \quad u_{\xi}=c\theta \quad \mbox{at} \quad \xi=-R.
\end{eqnarray}

Combining elementary computations above, we conclude that  any solution of \eqref{eq:1}--\eqref{eq:5} on an  interval $[0,R]$ must verify the following boundary value problem:
\begin{eqnarray}\label{eq:7}
\left\{
\begin{array}{llll}
u_{\xi\xi}-cu_{\xi}=-v^{\alpha}, & \Lambda v_{\xi\xi}-c v_{\xi} =v^{\alpha}& v>0 ,&  \xi\in(-R,0),\\
u=\theta, & u_{\xi}=c\theta & & \xi=-R,\\
u=1& u_{\xi}=0,& v=v_{\xi}=0, & \xi=0.
\end{array}
\right.
\end{eqnarray}
Therefore, the necessary condition for problem \eqref{eq:1}--\eqref{eq:5} to have a solution is the existence  of solution for problem \eqref{eq:7}.
We note that the existence of solution of  \eqref{eq:7} does not guarantee existence of solution for \eqref{eq:1}--\eqref{eq:5}. Indeed, one can easily verify  that
for any solution of  \eqref{eq:1}--\eqref{eq:5} we must have
\begin{eqnarray}
v(\xi)=1-a\exp\left(\frac{c}{\Lambda}(\xi+R)\right)\quad \xi\in(-\infty,-R),
\end{eqnarray}
for some constant $a\in(0,1)$. This constant, as follows from continuity conditions  \eqref{eq:2}, has to be chosen  from the overdetermined system:
\begin{eqnarray}
1-a=\lim_{\xi\to -R^+}v(\xi), \quad -ca=\Lambda\lim_{\xi\to -R^+}v_{\xi}(\xi),
\end{eqnarray}
which is not necessarily consistent. However, \cite[Theorem 1.1]{CMB21} guarantees that for any set of relevant parameters ( $\theta\in(0,1)$, $\alpha\in(0,1)$, $\Lambda\in(0,\infty)$) system \eqref{eq:1}--\eqref{eq:5} 
admits a solution. In view of this result,  problem \eqref{eq:7} admits a solution and uniqueness of the solution for problem \eqref{eq:7} implies uniqueness of the solution for \eqref{eq:1}--\eqref{eq:5}. 
The main result of this section is as follows:

\begin{theorem} \label{t:main}
Let $\theta\in(0,1)$, $\alpha\in(0,1)$, $\Lambda\in(0,\infty)$. Then, problem \eqref{eq:7} admits a unique solution.
\end{theorem}

Let us now show  that the proof of   Theorem \ref{t:main} reduces to the analysis of  the solution of a single second order ODE. 
First,  let us observe that integration of the equation for temperature $u$ in \eqref{eq:7} and integration of the same equation multiplied by  $\exp(-c\xi)$ taking into account boundary conditions give:
\begin{eqnarray}\label{eq:9}
c=\int_{-R}^0 v^{\alpha}(\xi)d\xi,
\end{eqnarray}
\begin{eqnarray}\label{eq:10}
c\theta=\int_{-R}^0 v^{\alpha}(\xi) \exp(-c(\xi+R)) d\xi.
\end{eqnarray}
Next we introduce the following rescaling:
\begin{eqnarray}\label{eq:11}
v(\xi)= A \tilde w(x),
\end{eqnarray}
where
\begin{eqnarray}\label{eq:12}
x:= c \xi, \quad A:= c^{-\frac{2}{1-\alpha}}.
\end{eqnarray}
Substituting   \eqref{eq:11}, \eqref{eq:12}
into the equation for concentration of reactant $v$ in \eqref{eq:7} and taking into account boundary conditions, we have:
\begin{eqnarray}\label{eq:13}
\left\{
\begin{array}{lll}
 \Lambda \tilde w^{\prime\prime}-\tilde w^{\prime} -\tilde w^{\alpha}=0, & \tilde w>0, &x\in (-cR,0),\\[2mm]
\tilde  w=0, & \tilde w^{\prime}=0,&  x=0,
\end{array}
\right.
\end{eqnarray}
where the prime denotes the derivative with respect to the variable $x$.

Consider now the first equation of \eqref{eq:13} extended on the half line, that is, the following problem:
\begin{equation}\label{eq:16}
\left\{
\begin{array}{lr}
\Lambda  w^{\prime\prime}- w^{\prime} - w^{\alpha}=0, & x<0,\\
w(0)=0, & w^{\prime}(0)=0.
\end{array}
\right.
\end{equation}
Properties of \eqref{eq:16} were studied in great detail in \cite[Section 4]{CMB21} via a topological approach, see also appendix below.
In particular, the following result comes directly from  Lemmas 4.7 and 4.8 of \cite{CMB21}, see appendix for more details.
\begin{lemma}\label{l:w}
There exists a unique function $w$ positive on $(-\infty,0)$   that verifies \eqref{eq:16}. 
\end{lemma}
It is clear that $w$ is smooth as long as it does not vanish. The function $w$ and its derivative can be extended by $0$ up to $0^-$. More precisely, it follows from \cite[Lemmas 6.5 \& 6.6]{CMB21} that the H\"older regularity of the function $w$ near $x=0$ reads:
$w\in C^{\infty}(-\infty,0))\cap C^{2+[\beta],\beta-[\beta]}(-\infty,0])$, $\beta=\frac{2\alpha}{1-\alpha}$, $0<\alpha<1$.	Moreover, $w$ is decreasing on $(-\infty,0]$.

As a consequence,  for an arbitrary $c>0$,  $R>0$ there  is a unique decreasing positive solution to \eqref{eq:13} and $\tilde w(x)=w(x)$  on $(-cR,0].$ 

In view of observations above we note that
in terms of rescaled variables  \eqref{eq:9}, \eqref{eq:10} read:
\begin{eqnarray}\label{eq:14}
c^{\frac{2}{1-\alpha}}=\int_{-cR}^0  w^{\alpha}(s)ds,
\end{eqnarray}
\begin{eqnarray}\label{eq:15}
 \theta c^{\frac{2}{1-\alpha}}=\int_{-cR}^0  w^{\alpha}(s) \exp(-(s+cR)) ds,
\end{eqnarray}
which in particular imply
\begin{eqnarray}
\theta=\frac{\int_{-cR}^0 w^{\alpha}(s) \exp(-(s+cR)) ds}{\int_{-cR}^0 w^{\alpha}(s) ds}.
\end{eqnarray}
Introduce the following function:
\begin{eqnarray}\label{eq:17}
\phi(x):=\frac{\int_{-x}^0 w^{\alpha}(s) \exp(-(s+x)) ds}{\int_{-x}^0 w^{\alpha}(s) ds}, \quad x\in(0,\infty).
\end{eqnarray}
The key observation that allows us to prove the main Theorem \ref{t:main} is the following. 
\begin{proposition}\label{p:main}
Let $w$ be the solution of \eqref{eq:16} and set
\begin{eqnarray}\label{eq:17}
\phi(x):=\frac{\int_{-x}^0 w^{\alpha}(s) \exp(-(s+x)) ds}{\int_{-x}^0 w^{\alpha}(s) ds}, \quad x\in(0,\infty).
\end{eqnarray}
The function $\phi$ defined by \eqref{eq:17} is strictly decreasing on $(0,\infty)$ and
\begin{eqnarray}
\lim_{x\to0}\phi(x)=1, \quad \lim_{x\to \infty} \phi(x)=0.
\end{eqnarray}
\end{proposition} 
Proof of Proposition \ref{p:main} is given in the next section. 
This proposition makes the proof of  Theorem \ref{t:main} given below elementary.
\begin{proof}[Proof of Theorem \ref{t:main}]
From \eqref{eq:14} and \eqref{eq:15} we have
\begin{eqnarray}\label{eq:18}
c=\left(\int_{-cR}^0 w^{\alpha}(s)ds\right)^{\frac{1-\alpha}{2}},
\end{eqnarray}
\begin{eqnarray}\label{eq:19}
\theta=\frac{\int_{-cR}^0 w^{\alpha}(s) \exp(-(s+cR)) ds}{\int_{-cR}^0 w^{\alpha}(s) ds}.
\end{eqnarray}

By strict monotonicity of $\phi$ (Proposition \ref{p:main}), there is a unique value $\sigma^*\in(0,\infty) $ such that $\phi(\sigma^*)=\theta.$
Thus, in \eqref{eq:18} we must have $cR=\sigma^*.$ This observation and uniqueness of the solution of \eqref{eq:16} imply
uniqueness of the speed $c=c^*$ with
\begin{eqnarray}\label{eq:20}
c^*=\left(\int_{-\sigma^*}^0 w^{\alpha}(s)ds\right)^{\frac{1-\alpha}{2}}.
\end{eqnarray}
The latter immediately implies uniqueness of position of the ignition interface  $R=R^*=\sigma^*/c^*.$
\end{proof}

\section{Proof of Proposition \ref{p:main}}

In this section we present a proof  of  Proposition \ref{p:main} which is  based on the following three lemmas.

\begin{lemma}\label{l:0} Function $\phi(x)$ defined by \eqref{eq:17} has the following properties:
\begin{eqnarray}\label{eq:a0}
\lim_{x\to 0}\phi(x)=1, \quad \lim_{x\to \infty} \phi(x)=0.
\end{eqnarray}
\end{lemma}
\begin{proof}
Let us first prove  the first equality in \eqref{eq:a0}.
Observe that
\begin{eqnarray}
e^{-x} \int_{-x}^0 w^{\alpha}(s)ds< \int_{-x}^0 w^{\alpha}(s)e^{-(s+x)} ds <\int_{-x}^0 w^{\alpha}(s)ds, \quad x>0.
\end{eqnarray}
Hence,
\begin{eqnarray}
e^{-x}<\phi(x)<1, \quad x>0.
\end{eqnarray}
Sending $x$ to zero in the inequality above we have the first equality in \eqref{eq:a0}.

Now let us  prove the second equality  in \eqref{eq:a0}.
Integrating the equation in \eqref{eq:16} and taking into account boundary conditions we have:
\begin{eqnarray}
-\Lambda w^{\prime}(-x)+w(-x)=\int_{-x}^0 w^{\alpha}(s)ds.
\end{eqnarray}
Since $w^{\prime}(-x)<0$ for $x>0$ we have:
\begin{eqnarray}
w(-x)<\int_{-x}^0 w^{\alpha}(s)ds.
\end{eqnarray}
Let 
\begin{eqnarray}
\rho(x)=\frac{\int_{-x}^0 w^{\alpha}(s)ds}{w^{\alpha}(-x)}.
\end{eqnarray}
In view of the inequality above we have:
\begin{eqnarray}
\rho(x)> w^{1-\alpha}(-x),
\end{eqnarray}
and thus   
\begin{eqnarray}\label{eq:a01}
\rho(x)\to \infty \quad \mbox{as} \quad x\to \infty, 
\end{eqnarray}
as $w(-x) \to \infty$ as $x\to\infty$.

Next observe that
\begin{eqnarray}
&&\int_{-x}^0 w^{\alpha}(s)e^{-(s+x)}ds=\int_{-x+\sqrt{\rho(x)}}^0w^{\alpha}(s)e^{-(s+x)}ds+\int^{-x+\sqrt{\rho(x)}}_{-x}w^{\alpha}(s)e^{-(s+x)}ds\le\nonumber\\
 &&e^{-\sqrt{\rho(x)}}\int_{-x}^0 w^{\alpha}(s)ds+w^{\alpha}(-x)\sqrt{\rho(x)}
\end{eqnarray}
Dividing  the expression above by $\int_{-x}^0 w^{\alpha}(s)ds$ and using the definitions of $\phi(x)$ and $\rho(x)$ we have:
\begin{eqnarray}
\phi(x)\le e^{-\sqrt{\rho(x)}} +\frac{1}{\sqrt{\rho(x)}}.
\end{eqnarray}
Sending $x\to \infty$ in this inequality and taking into account \eqref{eq:a01} we obtain the second equation in \eqref{eq:a0}.
\end{proof}

\begin{lemma}\label{l:1}
Let $Y$ be a  non-negative solution of 
\begin{eqnarray}\label{eq:a1}
\Lambda Y^{\prime\prime}-Y^{\prime}-Y^{\alpha}=0.
\end{eqnarray}
Assume there is $x_0\in\mathbb{R}$ such that 
\begin{eqnarray}\label{eq:a2}
Y(x_0)>0, \quad Y^{\prime}(x_0)<0, \quad Y^{\prime\prime}(x_0) Y(x_0)>\left( Y^{\prime}(x_0)\right)^2.
\end{eqnarray}
Then, 
\begin{eqnarray}\label{eq:a3}
\inf_{ x\in[x_0,\infty)}Y(x)>0. 
\end{eqnarray}
\end{lemma}

\begin{proof}

Let 
\begin{eqnarray}\label{eq:b1}
Y(x_0)=Y_0, \quad Y^{\prime}(x_0)=Y^{\prime}_0,
\end{eqnarray}
and assume that \eqref{eq:a2} hold.
Then, as follows from  \eqref{eq:a1} and the last inequality in \eqref{eq:a2}
\begin{eqnarray}\label{eq:b2}
\Lambda \frac{\left(Y_0^{\prime}\right)^2}{Y_0}-Y_0^{\prime}-Y_0^{\alpha}<\Lambda Y^{\prime\prime}(x_0)-Y_0^{\prime}-Y_0^{\alpha}=0,
\end{eqnarray}
and hence
\begin{eqnarray}\label{eq:b3}
\Lambda \left( \frac{Y_0^{\prime}}{Y_0}\right)^2-\left( \frac{Y_0^{\prime}}{Y_0}\right)-Y_0^{\alpha-1}<0.
\end{eqnarray}
Therefore,
\begin{eqnarray}\label{eq:b4}
\mu_{-}<\frac{Y_0^{\prime}}{Y_0}<0<\mu_{+},
\end{eqnarray}
where
\begin{eqnarray}\label{eq:b5}
\mu_{\pm}=\frac{1\pm \sqrt{1+4\Lambda Y_0^{\alpha-1}}}{2\Lambda},
\end{eqnarray}\label{eq:b6}
are roots of the quadratic equation
\begin{eqnarray}\label{eq:b7}
\Lambda \mu^2-\mu-Y_0^{\alpha-1}=0.
\end{eqnarray}

Next observe that a linear equation
\begin{eqnarray}\label{eq:b8}
\Lambda Z^{\prime\prime}-Z^{\prime}-Y_0^{\alpha-1} Z=0,
\end{eqnarray}
satisfying
\begin{eqnarray}\label{eq:b9}
Z(x_0)=Y_0, \quad Z^{\prime}(x_0)=-\gamma Y_0,
\end{eqnarray}
has the following solution
\begin{eqnarray}\label{eq:b10}
&&Z(x)=\left(\frac{2\Lambda Y_0}{\sqrt{1+4\Lambda Y_0^{\alpha-1}}} \right)\exp\left(\frac{x-x_0}{2\Lambda} \right) \cosh\left( \frac{\sqrt{1+4\Lambda Y_0^{\alpha-1}}}{2\Lambda}(x-x_0)\right)\times\nonumber\\
&&\left\{ |\mu_{-}|+\frac{1}{2\Lambda}-\left(\frac{1}{2\Lambda}+\gamma\right)\tanh\left( \frac{\sqrt{1+4\Lambda Y_0^{\alpha-1}}}{2\Lambda}(x-x_0)\right)\right\}.
\end{eqnarray}

In view that for $x\ge 0$   $\cosh(x)\ge 1,~0\le \tanh(x) <1$  we have from \eqref{eq:b10} that
\begin{eqnarray}\label{eq:b12}
Z(x)>Z_0=\frac{2\Lambda Y_0}{\sqrt{1+4\Lambda Y_0^{\alpha-1}}}\left( |\mu_{-}|-\gamma\right), \quad x\in[x_0,\infty).
\end{eqnarray}
Hence, $Z(x)$ is uniformly bounded away from zero on $[x_0,\infty)$ provided 
\begin{eqnarray}\label{eq:b13}
\gamma < |\mu_{-}|.
\end{eqnarray}

Let $x_M$ be the largest $x>x_0$ such that 
\begin{eqnarray}\label{eq:b14}
 0<Y(x)<Y_0 \quad x\in I=(x_0,x_M).
 \end{eqnarray}
Note that the existence of such $x_M$ is guaranteed since $Y(x_0)=Y_0$ and $Y^{\prime}(x_0)<0.$ 

We claim that  for any solution of \eqref{eq:a1} satisfying \eqref{eq:a2} 
the following inequalities hold
\begin{eqnarray}\label{eq:b15}
Y(x)>Z(x), \quad Y^{\prime}(x)>Z^{\prime}(x),
\end{eqnarray}
 for $x\in I.$ Here $Z$ is the solution of \eqref{eq:b8}, \eqref{eq:b9} with arbitrary
\begin{eqnarray}\label{eq:b16}
\gamma\in\left( \frac{|Y^{\prime}_0|}{Y_0}, |\mu_{-}| \right),
\end{eqnarray}
fixed.
Indeed, in view of the facts that $Y(x_0)=Z(x_0)$ and   $ Z^{\prime}(x_0)<Y^{\prime}(x_0)<0$ (which is guaranteed by our choice of parameter $\gamma$) we first observe
 that  \eqref{eq:b15}  holds
for $x\in(x_0,x_0+\delta)$ with $\delta>0$ sufficiently small. Assume now $x^*$ is a smallest value of $x\in I$ at which  \eqref{eq:b15} is violated. Consider two possibilities of how it can happen. In the first 
one we have
\begin{eqnarray}\label{eq:b17}
Y(x^*)=Z(x^*) \quad \mbox{and} \quad Y^{\prime}(x^*)>Z^{\prime}(x^*),
\end{eqnarray}
but then 
\begin{eqnarray}\label{eq:b18}
Y(x^*-\eps)<Z(x^*-\eps),
\end{eqnarray}
for $\eps>0$ sufficiently small. This contradicts the definition of $x^*$ and hence this situation is impossible.
The second possibility is
\begin{eqnarray}\label{eq:b19}
Y(x^*)\ge Z(x^*) \quad \mbox{and} \quad Y^{\prime}(x^*)=Z^{\prime}(x^*),
\end{eqnarray}
In this case we have
\begin{eqnarray}\label{eq:b20}
\Lambda Y^{\prime\prime}(x^*)=Y^{\prime}(x^*)+Y^{\alpha}(x^*)>Y^{\prime}(x^*)+Y_0^{\alpha-1} Y(x^*)\ge Z^{\prime}(x^*)+Y_0^{\alpha-1} Z(x^*)=\Lambda Z^{\prime\prime}(x^*),
\end{eqnarray}
and thus
\begin{eqnarray}\label{eq:b21}
Y^{\prime\prime}(x^*)>Z^{\prime\prime}(x^*).
\end{eqnarray}
But then
\begin{eqnarray}\label{eq:b22}
Y^{\prime}(x^*-\eps)<Z^{\prime}(x^*-\eps).
\end{eqnarray}
for $\eps>0$ sufficiently small, which again contradicts the definition of $x^*.$

Hence for $x\in I$ we have $Y_0>Y(x)\ge Z_0>0$ and thus by continuity
\begin{eqnarray}\label{eq:b23}
Z_0\le  Y(x) \le Y_0 \quad x\in[x_0,x_M].
\end{eqnarray}

Clearly at $x=x_M$  we necessarily have
\begin{eqnarray}\label{eq:b24}
Y(x_M)=Y_0, \quad Y^{\prime}(x_M)\ge 0.
\end{eqnarray}
By \eqref{eq:a1}  for $x>x_M$ we have
\begin{eqnarray}\label{eq:b25}
Y^{\prime}(x)=Y^{\prime}(x_M)e^{\frac{x-x_M}{\Lambda} }+\Lambda^{-1}\int_{x_M}^x Y^{\alpha}(s)e^{\frac{x-s}{\Lambda} }ds.
\end{eqnarray}
We now claim that $Y(x)>Y_0$ for $x>x_M.$ Indeed, by continuity $Y(x)>Y_0/2$ for $x\in[x_M, x_M+\delta]$ for some possibly small $\delta>0.$
Then by \eqref{eq:b25} we have
\begin{eqnarray}\label{eq:b25a}
Y^{\prime}(x)\ge\left(\frac{Y_0}{2}\right)^{\alpha} \left( e^{\frac{x-x_M}{\Lambda} }-1\right) \quad x \in (x_M,x_M+\delta].
\end{eqnarray}
This implies that
\begin{eqnarray}
Y^{\prime}(x)>0 \quad \mbox{on} \quad (x_M,x_M+\delta],
\end{eqnarray}
and hence
\begin{eqnarray}
Y(x)>Y_0 \quad \mbox{on} \quad (x_M,x_M+\delta].
\end{eqnarray}
Thus,
\begin{eqnarray}\label{eq:b25aa}
Y(x_M+\delta)>Y_0, \quad Y^{\prime}(x_M+\delta)>0.
\end{eqnarray}
Now assume that $x_M^*$ is the first point greater than $x_M+\delta$ such that $Y^{\prime}(x_M^*)=0,$
but by \eqref{eq:b25} and \eqref{eq:b25aa}
\begin{eqnarray}\label{eq:b25a}
Y^{\prime}(x_M^*)\ge Y^{\prime}(x_M+\delta) e^{\frac{x_M^*-(x_M+\delta)}{\Lambda}}  +Y_0^{\alpha}\left(  e^{\frac{x_M^*-(x_M+\delta)}{\Lambda}} -1 \right) >0,
\end{eqnarray}
and hence $Y(x)$ is strictly increasing on $(x_M,\infty).$ Thus
\begin{eqnarray}\label{eq:b26}
Y(x)>Y_0 \quad x\in(x_M,\infty).
\end{eqnarray}
Combining \eqref{eq:b23} and \eqref{eq:b26} we have 
\begin{eqnarray}
Y(x)\ge Z_0>0 \quad x\in[x_0,\infty),
\end{eqnarray}
that imply \eqref{eq:a3}  and thus completes the proof.

\end{proof}

\begin{corollary}\label{cor:1}
It follows immediately from Lemma \ref{l:1} that the solution of \eqref{eq:16} must satisfy
\begin{eqnarray}\label{eq:a4}
w^{\prime\prime}(x) w(x)\le \left( w^{\prime}(x)\right)^2, \quad \forall x\in(-\infty,0).
\end{eqnarray}
Indeed, if there was a point $x_0<0$ where \eqref{eq:a4} didn't hold, then it would imply that
$w(x)>0$ for $x\in[x_0,0].$ Thus such $w$  would  not satisfy boundary condition $w(0)=0$ and hence would not solve \eqref{eq:16} which gives a contradiction.
\end{corollary}

We note that formula \eqref{eq:a4} has a clear geometric interpretation. Consider problem \eqref{eq:16} as a dynamical system on the plane. Introducing  change of variables
\begin{eqnarray}\label{eq:cl1}
t=x, \quad q(t)=w(x), \quad p(t)=w'(x), 
\end{eqnarray}
we obtain   the following system of two first order ODE's:
\begin{eqnarray}\label{eq:cl2}
\left\{
\begin{array}{l}
\dot{q}(t)=p(t),\\
\Lambda \dot{p}(t)=p(t)+q^{\alpha}(t).
\end{array}
\right.
\end{eqnarray}
where dot stands for the derivative with respect to $t$.
This problem is considered for $t<0$ in the quadrant $Q=\{{q} \geqslant0, { p} \leqslant0\}.$
Rewriting  $p$ and $q$  in polar coordinates we have
\begin{eqnarray}\label{eq:cl3}
q(t):=r(t)\cos(\vartheta(t)), \quad p(t):= r(t) \sin(\vartheta(t)), 
\end{eqnarray}
where, $t\in(-\infty,0),\; r>0,\; \theta \in (-\pi/2,0)$.
In terms of the new variables condition \eqref{eq:a4}  is equivalent to 
\begin{eqnarray}\label{eq:cl3a}
\dot{\vartheta}(t)\le 0.
\end{eqnarray}
We present derivation of this formula in the  appendix.

\begin{lemma}\label{l:2}
Let $w$ be the solution of \eqref{eq:16} and set
\begin{eqnarray}\label{eq:a5}
f(x):=\left\{
\begin{array}{ll}
w^{\alpha}(x), & x\le 0,\\
0,& x>0.
\end{array}
\right.
\end{eqnarray}
For any  fixed  $y>0,$ the function
\begin{eqnarray}\label{eq:a6}
\psi(x,y)=\frac{\int_{-x+y}^{\infty} f(s)ds}{\int_{-x}^{\infty} f(s)ds},
\end{eqnarray}
is non-decreasing function of $x$ on $(0,\infty)$ and strictly increasing on $(y,\infty).$ 
\end{lemma}

\begin{proof}
First we observe that 
\begin{eqnarray}\label{eq:a7}
\psi(x,y)\equiv 0 \quad \mbox{for} \quad x\le y.
\end{eqnarray}
Hence, $\psi(x,y)$ is non-decreasing in $x$ for $x\in(0,y].$

For $x>y$ we have,
\begin{eqnarray}\label{eq:a8}
\frac{\partial}{\partial x} \psi(x,y)= \frac{f(y-x) \int_{-x}^\infty f(s) ds-f(-x) \int_{y-x}^\infty f(s) ds}{\left(\int_{-x}^\infty f(s) ds\right)^2}.
\end{eqnarray}
Thus to show that $\psi(x,y)$ is strictly increasing in $x$ for $x\in(y,\infty),$ we need to prove the following inequality
\begin{eqnarray}\label{eq:a9}
\frac{\int_{y-x}^\infty f(s) ds}{f(y-x)}<\frac{\int_{-x}^\infty f(s) ds}{f(-x)} \qquad x\in(y,\infty).
\end{eqnarray}

Define 
\begin{eqnarray}\label{eq:a10}
g(x):=\frac{f^{\prime}(x)}{f(x)}=\alpha \frac{w^{\prime}(x)}{w(x)} \quad x\in(-\infty,0).
\end{eqnarray}
Observe that 
\begin{eqnarray}\label{eq:a11}
g^{\prime}(x)=\frac{\alpha}{w^2(x)}\left(w^{\prime\prime}(x)w(x) -\left(w^{\prime}(x)\right)^2 \right),
\end{eqnarray}
and thus by Corollary \ref{cor:1} 
\begin{eqnarray}\label{eq:a12}
g^{\prime}(x) \le 0 \quad \mbox{on} \quad (-\infty,0).
\end{eqnarray}
This in particular implies  that for $x>y>0$ we have
\begin{eqnarray}\label{eq:a13}
g(s+y-x) \le g(s-x) \quad s\in(0,x-y).
\end{eqnarray}
Consequently,
\begin{eqnarray}\label{eq:a14}
\int_0^s g(\tau+y-x)d\tau \le \int_0^s g(\tau-x)d\tau.
\end{eqnarray}
By the definition of the function $g$ (see Eq. \eqref{eq:a10})  we have
\begin{eqnarray}\label{eq:a15}
&&\int_0^s g(\tau+y-x)d\tau= \int_0^s \frac{\frac{d}{d\tau} f(\tau+y-x) }{ f(\tau+y-x)}d\tau= \int_0^s \left( \frac{d}{d\tau} \log  f(\tau+y-x)\right) d\tau=\log\left( \frac{f(s+y-x)}{ f(y-x)}\right),\nonumber \\
&&\int_0^s g(\tau-x)d\tau= \int_0^s \frac{\frac{d}{d\tau} f(\tau-x) }{ f(\tau-x)}d\tau= \int_0^s \left( \frac{d}{d\tau} \log  f(\tau-x)\right) d\tau=\log\left( \frac{f(s-x)}{ f(-x)}\right).
\end{eqnarray}\label{eq:a16}
Therefore, by \eqref{eq:a14} and \eqref{eq:a15} we have that for $x>y>0$ and $s\in(0,x-y)$
\begin{eqnarray}\label{eq:a17}
\log\left( \frac{f(s+y-x)}{ f(y-x)}\right)\le \log\left( \frac{f(s-x)}{ f(-x)}\right),
\end{eqnarray}
that is
\begin{eqnarray}\label{eq:a18}
 \frac{f(s+y-x)}{ f(y-x)} \le \frac{f(s-x)}{ f(-x)}.
\end{eqnarray}
Using the facts that $f(x)\equiv 0$ for $x>0$  and $f(x)>0$ for $s\in(-y,0)$ for $y>0$  we observe that 
\begin{eqnarray}\label{eq:a19}
&& \int_0^{x-y} f(s+y-x)ds =\int_{y-x}^0 f(s)ds=\int_{y-x}^{\infty} f(s)ds, \nonumber \\
&& \int_0^{x-y} f(s-x)ds =\int_{-x}^{-y}  f(s)ds<\int_{-x}^{0} f(s)ds=\int_{-x}^{\infty} f(s)ds.
\end{eqnarray}
Hence integrating  \eqref{eq:a18} in $s$ from $0$ to $x-y$ and taking into account \eqref{eq:a19} we conclude that \eqref{eq:a9} holds for arbitrary $x>y>0$ and hence $\psi(x,y)$ is  indeed strictly increasing in $x$  for $x>y>0.$ 
\end{proof}
Now we are ready to give a proof of the proposition

\begin{proof}[Proof of Proposition \ref{p:main}]

Let 
\begin{eqnarray}\label{eq:a20}
\chi(x,y)=1-\psi(x,y)=\frac{\int_{-x}^{y-x} f(s)ds}{\int_{-x}^{\infty} f(s)ds}.
\end{eqnarray}
Observe that
\begin{eqnarray}\label{eq:a21}
&& \int_0^{\infty}\left\{\int_{-x}^{y-x} f(s)ds\right\} e^{-y}dy=\int_0^{\infty}\left\{ \int_{0}^{y} f(s-x) e^{-y} ds\right\} dy=\int_0^{\infty}  \left\{\int_{s}^{\infty} f(s-x) e^{-y} dy\right\} ds \nonumber \\
&&=\int_0^{\infty}  \left\{\int_{s}^{\infty}e^{-y} dy\right\}  f(s-x)  ds= \int_0^{\infty}     f(s-x)  e^{-s} ds=\int_{-x}^{\infty}     f(s)  e^{-(s+x)} ds.
\end{eqnarray}
Hence,
\begin{eqnarray}\label{eq:a22}
\int_0^{\infty} \chi(x,y) e^{-y}dy=\frac{\int_{-x}^{\infty}  f(s)  e^{-(s+x)} ds}{\int_{-x}^{\infty} f(s)ds}=\frac{\int_{-x}^{0}  w^{\alpha}(s)  e^{-(s+x)} ds}{\int_{-x}^{0} w^{\alpha}(s)ds}=\phi(x).
\end{eqnarray}
In computation above we used the definition of $f$ (see Eq. \eqref{eq:a5} ) and $\psi$ (see Eq.\eqref{eq:a6}).

In view of the statement of Lemma \ref{l:2}, $\chi(x,y)$ is non-increasing function of  $x$ on $(0,\infty)$ and  is strictly decreasing function of $x$ on $[y,\infty).$
Consequently, $\int_0^{\infty} \chi(x,y) e^{-y}dy$ is a strictly decreasing function, and therefore,  $\phi(x)$ is strictly decreasing as follows from \eqref{eq:a22}.

\end{proof}

\section{Dependency of the speed of propagation on parameters: asymptotics and numerics.}

In the previous sections, we established the uniqueness of a solution for problem \eqref{eq:7}. Our main result states that for an arbitrary fixed $\theta\in(0,1),$ $\Lambda\in(0,\infty)$ and $\alpha\in(0,1)$ there exists a unique pair $(c^*,R^*)$   for which this problem
admits a solution. This pair represents  the velocity of propagation and width of the reaction zone for this set of parameters. 
    The natural question is then to investigate the dependency of these quantities on the parameters of the problem,
that is, dependencies $c^*(\theta,\Lambda,\alpha)$ and $R^*(\theta,\Lambda,\alpha)$. Of a particular interest is how the velocity of propagation changes with the parameters of the problem as this quantity is the main characteristic of the flame front.

As we showed in the previous sections, $(c^*,R^*)$ are uniquely determined from the solution of problem \eqref{eq:16} that represents scaled distribution of the deficient reactant over the reaction zone.
Hence, we start with a discussion of qualitative properties of the solution of \eqref{eq:16} . While the solution  of this problem can not be obtained in the closed form, it can be computed numerically.
 Figure \ref{f:3} depicts function $w$ for different values of $\alpha$ and $\Lambda=1.$
\begin{figure}[h!]
\centering \includegraphics[width=4.in]{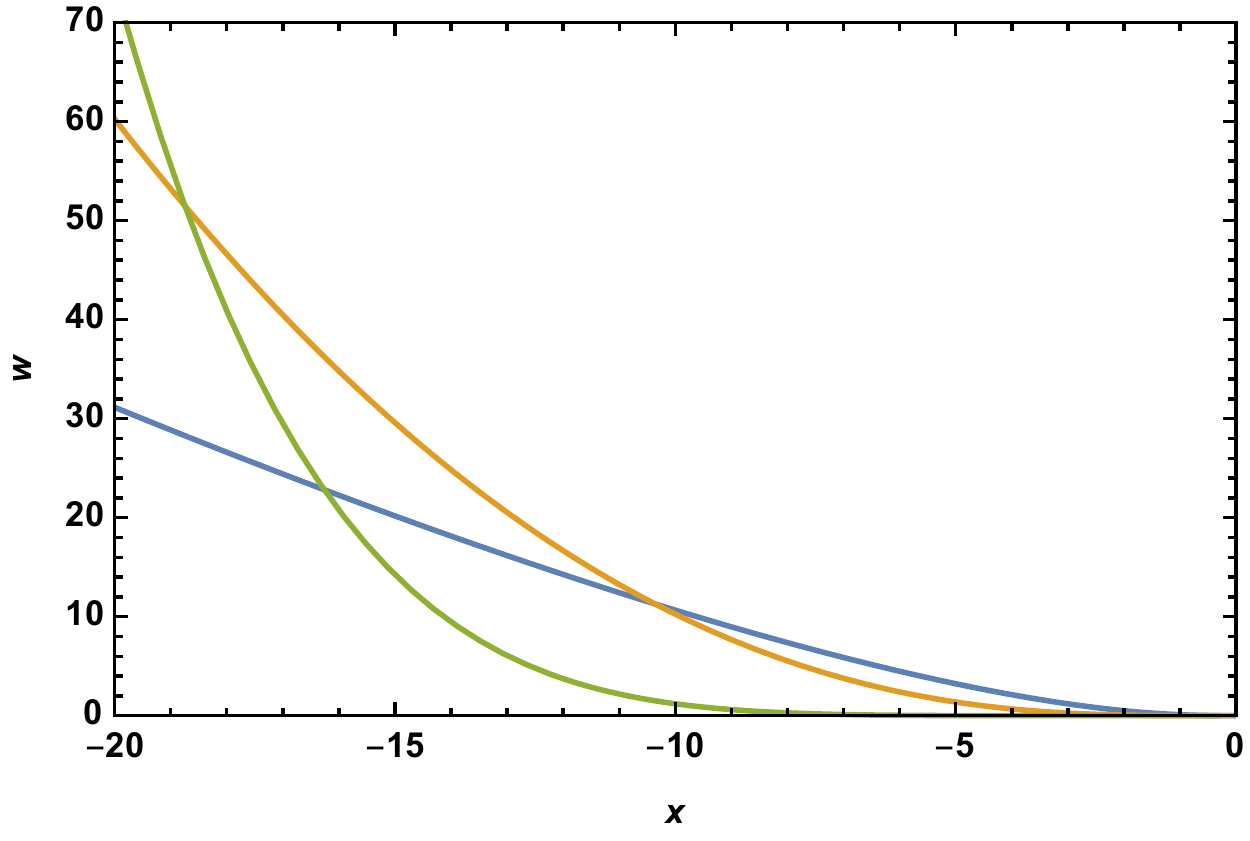}
\caption{Numerical solutions of problem \eqref{eq:16} for $\alpha=1/4$ (blue), $\alpha=1/2$ (orange) and $\alpha=3/4$ (green). 
   }
\label{f:3} 
\end{figure}
We note that solutions of  \eqref{eq:16} for $\Lambda\ne 1$ can be obtained from solutions of this equation with $\Lambda=1$ by rescaling. Indeed, one can verify by the direct substitution that
\begin{eqnarray}\label{eq:21a}
w(x\vert \Lambda,\alpha)={\Lambda}^\frac{1}{1-\alpha}w\left(\frac{x}{\Lambda}\Big\vert\Lambda=1, \alpha\right).
\end{eqnarray}
Observe that  solutions for \eqref{eq:16}  for different values of $\alpha$ are not ordered. Indeed, for fixed  $|x|>0$ sufficiently small,  $w(x\vert \Lambda, \alpha)$ is a decreasing function of $\alpha.$ Whereas for $|x|$ sufficiently large,
$w(x\vert \Lambda, \alpha)$ is increasing function of $\alpha$.
Formal considerations  (which can be made rigorous) show that asymptotic behavior of the solution of \eqref{eq:16} is as follows:
\begin{eqnarray}\label{eq:21}
&&w(x)=\left[\frac{(1-\alpha)^2}{2\Lambda(1+\alpha)}\right]^\frac{1}{1-\alpha} (-x)^\frac{2}{1-\alpha}(1+o(1)) \quad |x|\ll 1, \nonumber\\
&& w(x)=\left({1-\alpha}\right)^\frac{1}{1-\alpha} (-x)^\frac{1}{1-\alpha} (1+o(1)) \quad |x|\gg 1.
\end{eqnarray}
That is, as $\alpha$ increases, $w$ gets flatter and flatter near the origin and grows faster and faster for large values of $|x|$. Another important observation is that the solution \eqref{eq:16}  
near the origin is heavily influenced by the specific value of the Lewis number as in this region the solution is obtained (in the first approximation) by balancing diffusion and reaction
whereas the transport term is negligible. In contrast, far from the origin, the key ingredients are transport and reaction while diffusion is negligible in this region. Consequently, the solution of 
\eqref{eq:16}   is essentially independent of the Lewis number when $|x|\gg 1.$

Let us also discuss the limiting behavior of the function $w$ as $\alpha\to 0$ and $\alpha\to 1$.  It is straightforward to verify that as $\alpha\to 0$ the solution of \eqref{eq:16} approaches, on a compact sets, to the limiting profile
\begin{eqnarray}
w_0(x)=-x+\Lambda\left(\exp \left(\frac{x}{\Lambda}\right)-1\right),
\end{eqnarray}
that verifies the limiting problem
\begin{equation}
\left\{
\begin{array}{lr}
\Lambda  w_0^{\prime\prime}- w_0^{\prime} - 1=0, & x<0,\\
w_0(0)=0, & w_0^{\prime}(0)=0.
\end{array}
\right.
\end{equation}
We now claim that as $\alpha\to 1$ the function $w$ approaches zero on  compact sets. Indeed, let $ \eta=w-\Lambda w^{\prime}$. In view that $w(x)>0$ and $w^{\prime}(x)<0$ on $(-\infty ,0)$ and $w(0)=w^{\prime}(0)=0,$ we have
$\eta(x)>0$ on $(-\infty,0)$ as $\eta(0)=0.$ By \eqref{eq:16} we then have
\begin{eqnarray}
\left\{
\begin{array}{ll}
-\eta^{\prime}=w^{\alpha}, & x<0,\\
\eta(0)=0.
\end{array}
\right. 
\end{eqnarray}
Taking into account that $\eta,w$ are positive on $(-\infty,0)$ we the have
\begin{eqnarray}
-\eta^{\prime}=w^{\alpha}\le \eta^{\alpha}
\end{eqnarray}
Integrating this inequality and taking into account the initial condition we have
\begin{eqnarray}
\eta(x)\le \left[ (1-\alpha)(-x)\right]^\frac{1}{1-\alpha}, \quad x\le 0
\end{eqnarray}
Since $w<\eta$ we conclude
\begin{eqnarray}\label{eq:wup}
w(x)\le \left[ (1-\alpha)(-x)\right]^\frac{1}{1-\alpha}, \quad x\le 0.
\end{eqnarray}
Fixing $x<0$ and taking a limit $\alpha\to 1$ in the inequality above, we conclude that $w\to 0$ as $\alpha\to 0$ on compact sets.

Now we turn to the evaluation of  the pair $c^*(\theta,\Lambda,\alpha)$, $R^*(\theta,\Lambda,\alpha)$.
An algorithm for finding this  pair is rather straightforward. First, we fix $\Lambda$ and $\alpha$ and solve problem \eqref{eq:16}. Then, we follow the procedure outlined in the proof of the main Theorem \ref{t:main}. Namely,
using the solution of \eqref{eq:16}, we evaluate $\phi(x)$ given by \eqref{eq:17} and a function $\zeta(x)$ given as:
\begin{eqnarray}\label{eq:zeta}
\zeta(x)
=\left(\int_{-x}^0 w^{\alpha}(s)ds\right)^{\frac{1-\alpha}{2}}.
\end{eqnarray}
We next fix $\theta$ and find a unique number $\sigma^*>0$ such that
\begin{eqnarray}
\phi(\sigma^*)=\theta .
\end{eqnarray}
The existence and uniqueness of such  $\sigma^*$ follows from Proposition \ref{p:main}.
Hence as follows from \eqref{eq:19}
\begin{eqnarray}
c^*=\zeta(\sigma^*),
\end{eqnarray}
and then by \eqref{eq:20}
\begin{eqnarray}
R^*=\sigma^*/c^*.
\end{eqnarray}

The function  $\phi(x)$  is decreasing, and the function $\zeta(x)$ is increasing  on $x\in(0,\infty).$  
Moreover, by \eqref{eq:21a} we have
\begin{eqnarray}
\quad \zeta(x\vert \Lambda, \alpha)=\sqrt{\Lambda} \zeta\left( \frac{x}{\Lambda}\big\vert\Lambda= 1, \alpha\right).
\end{eqnarray}
Let us now discuss the dependency of $\phi,\zeta$ on parameter $\alpha$. Several profiles of  functions $\phi$ and $\zeta$ for several values of $\alpha$ and $\Lambda=1$ are depicted in Figures \ref{f:pz}. For a fixed $x\in(0,\infty),$ the function $\phi$ appears to be increasing in $\alpha$ and the 
function $\zeta$ decreasing in $\alpha.$ 
\begin{figure}[h]
\hspace{-.5cm}
\begin{minipage}[c]{0.4\textwidth}%
\includegraphics[width=3.5in]{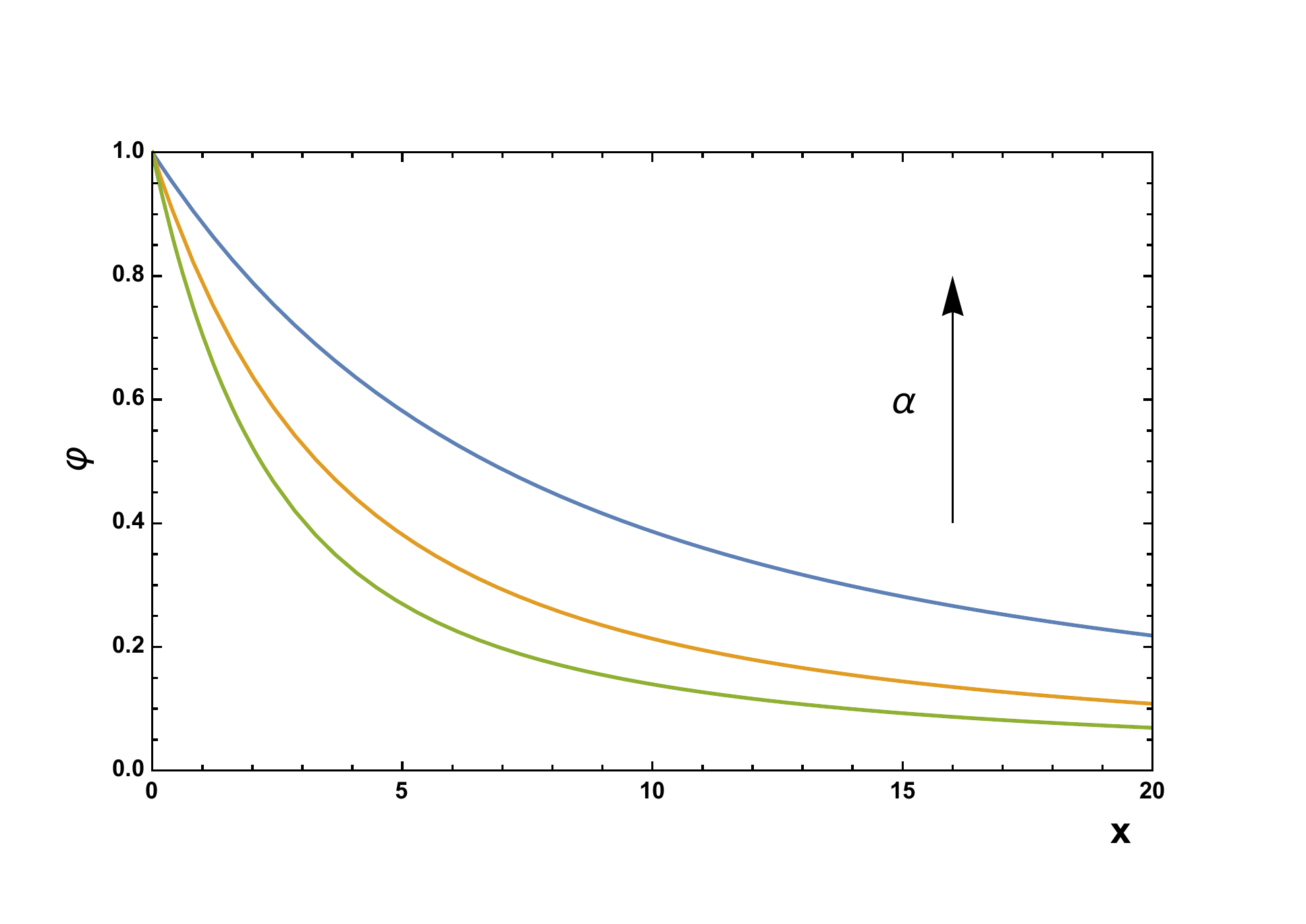}
\end{minipage}\hspace{1.2cm} %
\hspace{1cm}
\begin{minipage}[c]{0.4\textwidth}%
 \includegraphics[width=3.5in]{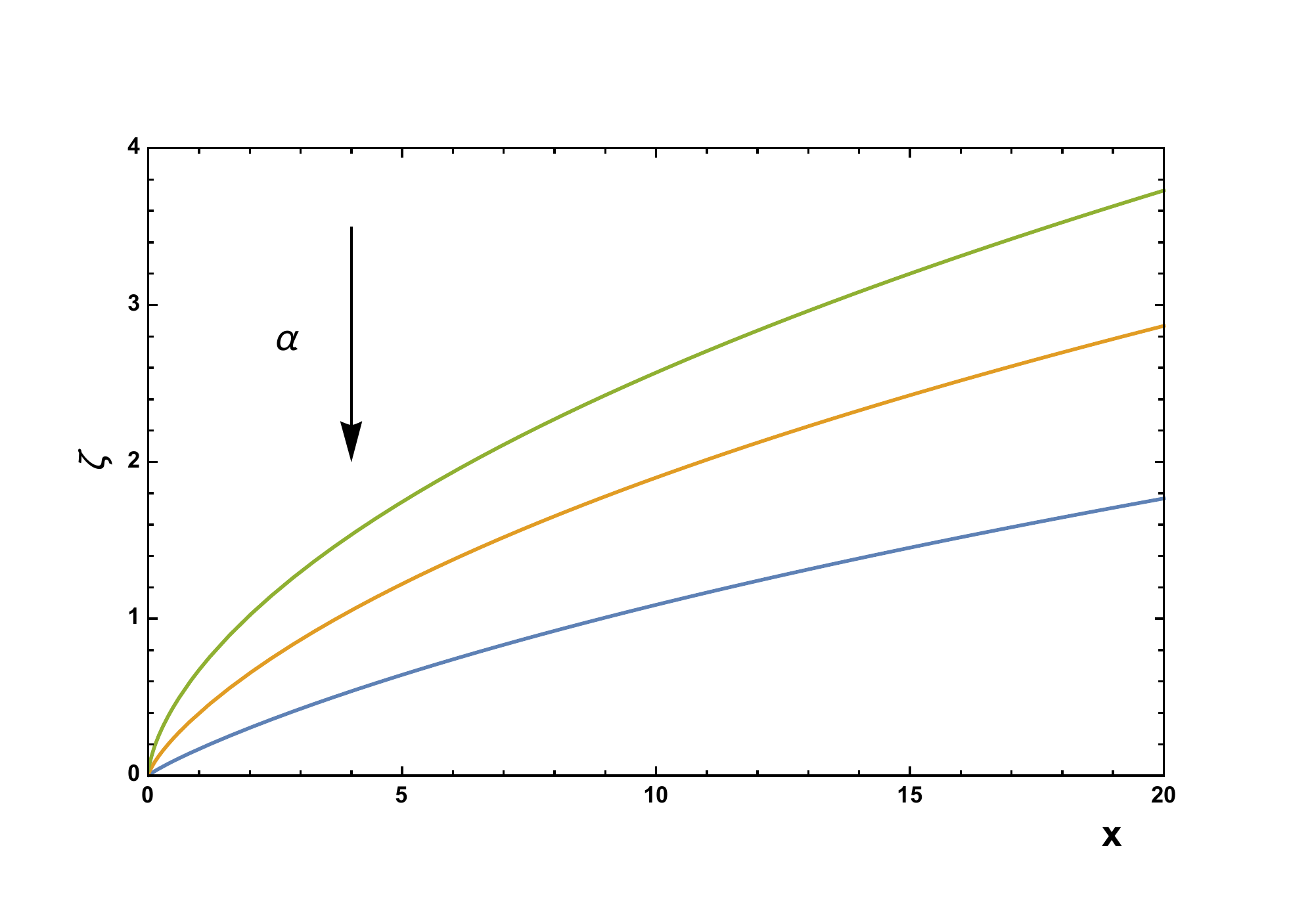} %
\end{minipage}\caption{{Functions $\phi$ and $\zeta$ for $\alpha=0.75$ (blue), $\alpha=0.5$ (orange) and $\alpha=0.25$ (green) and $\Lambda=1$. The arrow indicates direction of increase of $\alpha$.
}}
\noindent \label{f:pz} 
\end{figure}
It is also easy to check that on compact sets  $\phi(x)\to \frac{1-e^{-x}}{x}$ and $\zeta(x)\to \sqrt{x}$  as $\alpha\to 0$ which follows directly from the fact that $w$ approaches to $w_0$ in this limit.
We now claim that  $\phi(x)$ approaches  unity and $\zeta(x)$ approaches zero on any fixed compact subset of $x\in(0,\infty)$ as $\alpha\to 1.$
To see that let us observe first that by \eqref{eq:16} and \eqref{eq:17} after integration by parts we have:
\begin{eqnarray}
\phi(x)=1-\tilde \phi(x),
\end{eqnarray}
with
\begin{eqnarray}
\tilde \phi(x)=\frac{\Lambda+(1-\Lambda)\int_{-x}^0 w(s) \exp(-(s+x))ds/w(-x)}{1+\Lambda|w^{\prime}(-x)/w(-x)|}\le \frac{\Lambda+|1-\Lambda|}{1+\Lambda |w^{\prime}(-x)/w(-x)|}, \quad x>0,
\end{eqnarray}
where the last inequality follows from the monotonicity of $w$.
Next by \eqref{eq:16} and  \eqref{eq:a4} 
we obtain
\begin{eqnarray}
w^{\prime\prime}(x) w(x)=\frac{1}{\Lambda}\left(w^{\prime}(x)+w^{\alpha}(x)\right) w(x)\le( w^{\prime}(x))^2, \quad x<0.
\end{eqnarray}
Dividing the expression above by $w^2$ and taking into account that $w^{\prime}<0$ we have
\begin{eqnarray}
\Lambda \left( \frac{w^{\prime}(x) }{w(x)} \right)^2 +\left| \frac{w^{\prime}(x) }{w(x) }\right|\ge \frac{1}{w^{1-\alpha}(x)}, \quad x<0.
\end{eqnarray}
Combining this observation with \eqref{eq:wup} we have
\begin{eqnarray}
\Lambda \left( \frac{w^{\prime}(x) }{w(x)} \right)^2 +\left| \frac{w^{\prime}(x) }{w(x) }\right|\ge \frac{1}{(1-\alpha)(-x)}, \quad x<0.
\end{eqnarray}
Fixing $x<0$ and taking a limit as $\alpha\to1$ in the expression above, we conclude that $|w^{\prime}/w| \to \infty$ as $\alpha\to 1$ on compact sets.
This observation implies that $\tilde \phi(x)\to 0$ as $\alpha\to 1$ and hence $\phi(x)\to 1$ in this limit.
Finally by \eqref{eq:wup} and \eqref{eq:zeta} we have
\begin{eqnarray}
\zeta(x)\le \sqrt{(1-\alpha)} \sqrt{x}.
\end{eqnarray}
Taking a limit as $\alpha\to1$ for $x>0$ fixed we obtain $\zeta\to 0$ as $\alpha\to 1$ on compact sets.

We now will discuss the dependency of $\phi,\zeta$ of $\alpha$. Figure \ref{f:pz2} depicts  functions $\phi$ and $\zeta$ for several values of $\Lambda$ and  $\alpha=1/2.$    For a fixed $x\in(0,\infty),$ the function $\phi$ is  increasing in $\Lambda,$ and the 
function $\zeta$ is decreasing in $\Lambda.$ 

\begin{figure}[h]
\hspace{-.5cm}
\begin{minipage}[c]{0.4\textwidth}%
\includegraphics[width=3.5in]{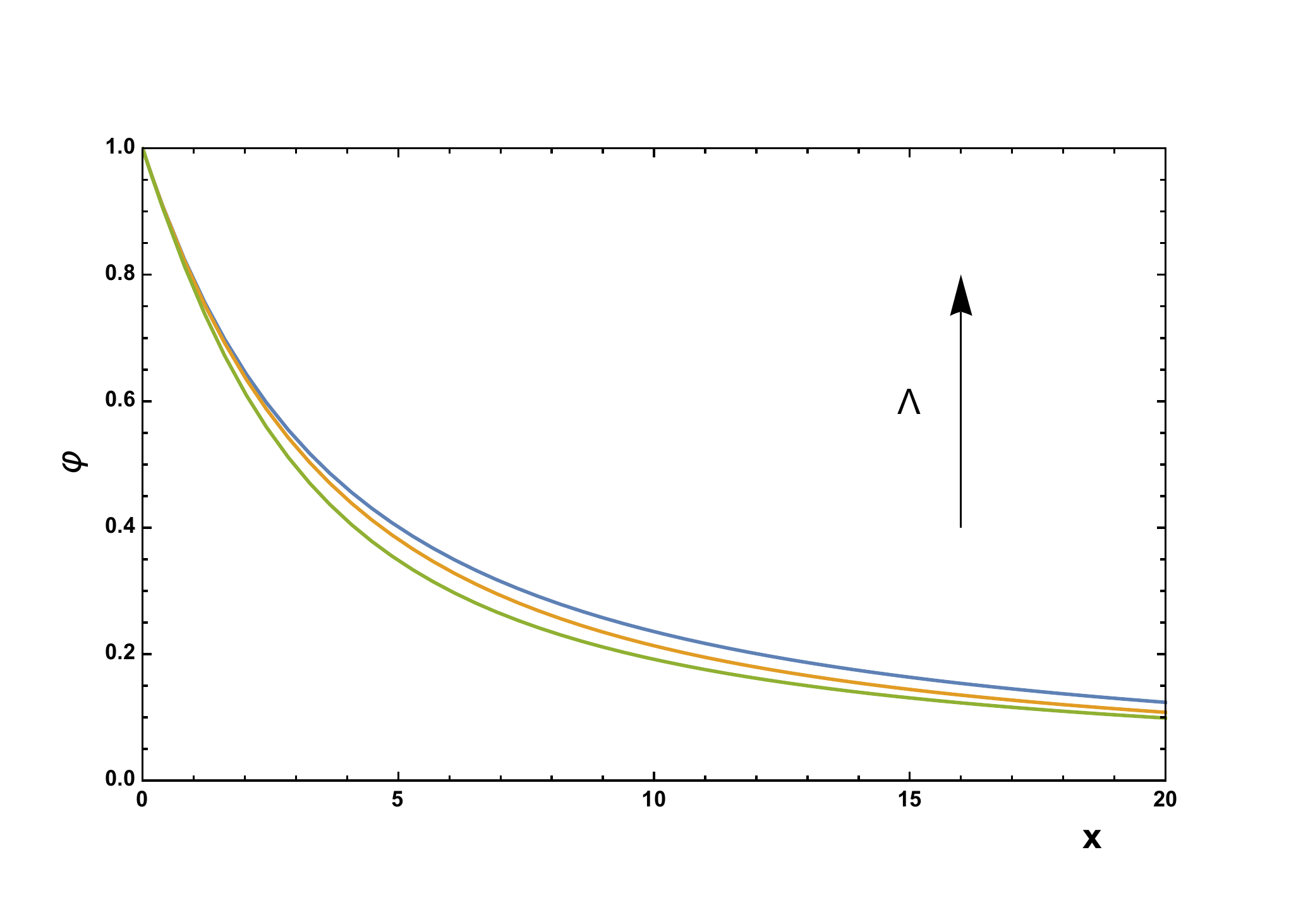}
\end{minipage}\hspace{1.2cm} %
\hspace{1cm}
\begin{minipage}[c]{0.4\textwidth}%
 \includegraphics[width=3.5in]{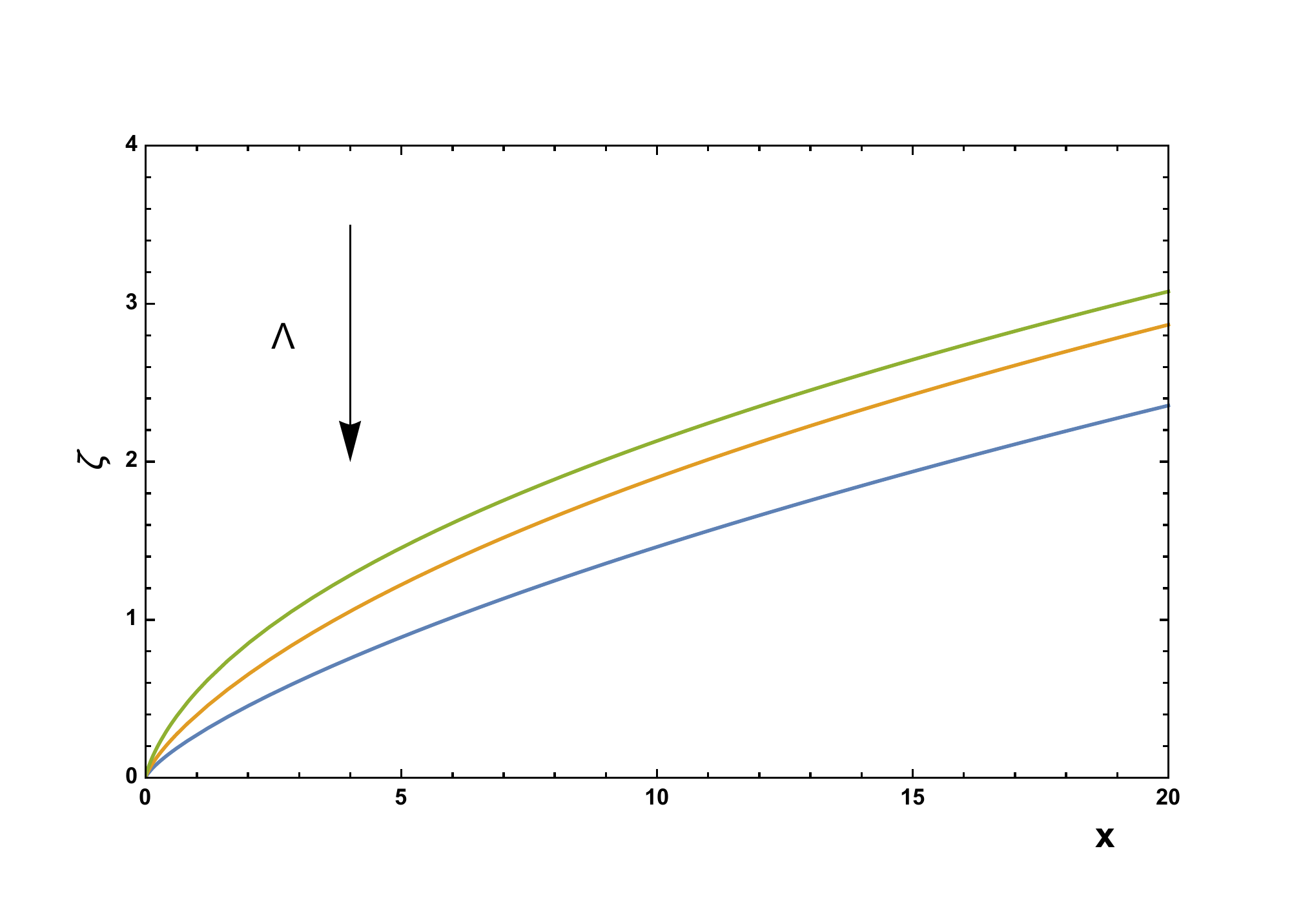} %
\end{minipage}\caption{{Functions $\phi$ and $\zeta$ for $\Lambda=5$ (blue), $\Lambda=1$ (orange) and $\Lambda=0.2$ (green) and $\alpha=1/2$. The arrow indicates direction of increase of $\Lambda$.
}}
\noindent \label{f:pz2} 
\end{figure}
One obvious observation that follows from the monotonicity  of $\phi$ is that $\sigma^*$ is a decreasing function of $\theta$ which together with the monotonicity of $\zeta$ immediately implies  that $c^*$ is a decreasing function of $\theta$.
This result is very natural in a physical context of the problem as an increase in ignition temperature decreases the speed of propagation.

Let us now discuss the behavior of $\phi(x), \zeta(x)$   near the origin which will allow us to obtain asymptotic expressions for $(c^*,R^*)$ for ignition temperatures near unity.
This regime is of particular interest as flame fronts are  known to become unstable when ignition temperature approaches one.  Discussion of flame front instabilities in this regime for the cases of zero and first order kinetics
can be found in \cite{cnf15}.

The behavior of functions $\phi(x)$ and $\zeta(x)$  for  small values of $x$ can be reconstructed from
asymptotic formula \eqref{eq:21}. After rather tedious but straightforward computations, we have:
\begin{eqnarray}\label{eq:as1}
\phi(x\vert \Lambda, \alpha)\approx 1-\left(\frac{1-\alpha}{2}\right) x, \quad \zeta(x\vert \Lambda, \alpha)=(2\Lambda)^{-\frac{\alpha}{2}}\left( \frac{(1-\alpha)^\frac{1+\alpha}{2}}{\sqrt{1+\alpha}} \right) x^\frac{1+\alpha}{2}  \quad \mbox{for} \quad x\ll 1.
\end{eqnarray}
This observation immediately implies that for  ignition temperatures close to unity we have:
\begin{eqnarray}\label{eq:as2}
&& c^*(\theta,\Lambda,\alpha)\simeq \sqrt{\frac{2}{1+\alpha}}\Lambda^{-\frac{\alpha}{2}}(1-\theta)^\frac{1+\alpha}{2},\quad R^*(\theta,\Lambda,\alpha)\simeq  \frac{\sqrt{2(1+\alpha)}}{1-\alpha}\Lambda^{\frac{\alpha}{2}}(1-\theta)^{\frac{1-\alpha}{2}}, \quad |1-\theta|\ll 1.
\end{eqnarray}
Direct verification shows  that, in this regime, $c^*$ is a decreasing function of  both $\alpha$ and $\Lambda,$ whereas $R^*$ is an increasing function of both of these parameters. Moreover, $R^*$ is decreasing function of $\theta$. Figure \ref{f:t1}  depicts
dependency of $c^*$ and $R^*$ on $\alpha$ for several values of $\Lambda$ with $\theta=0.98,$ and  Figure \ref{f:t13d} shows the dependency of the velocity $c^*$ on  $\alpha$ and $\Lambda$.
These figures were generated using asymptotic formulas \eqref{eq:as2} which are extremely close to their numerical counterparts .
\begin{figure}[h]
\hspace{-.5cm}
\begin{minipage}[c]{0.4\textwidth}%
\includegraphics[width=3.5in]{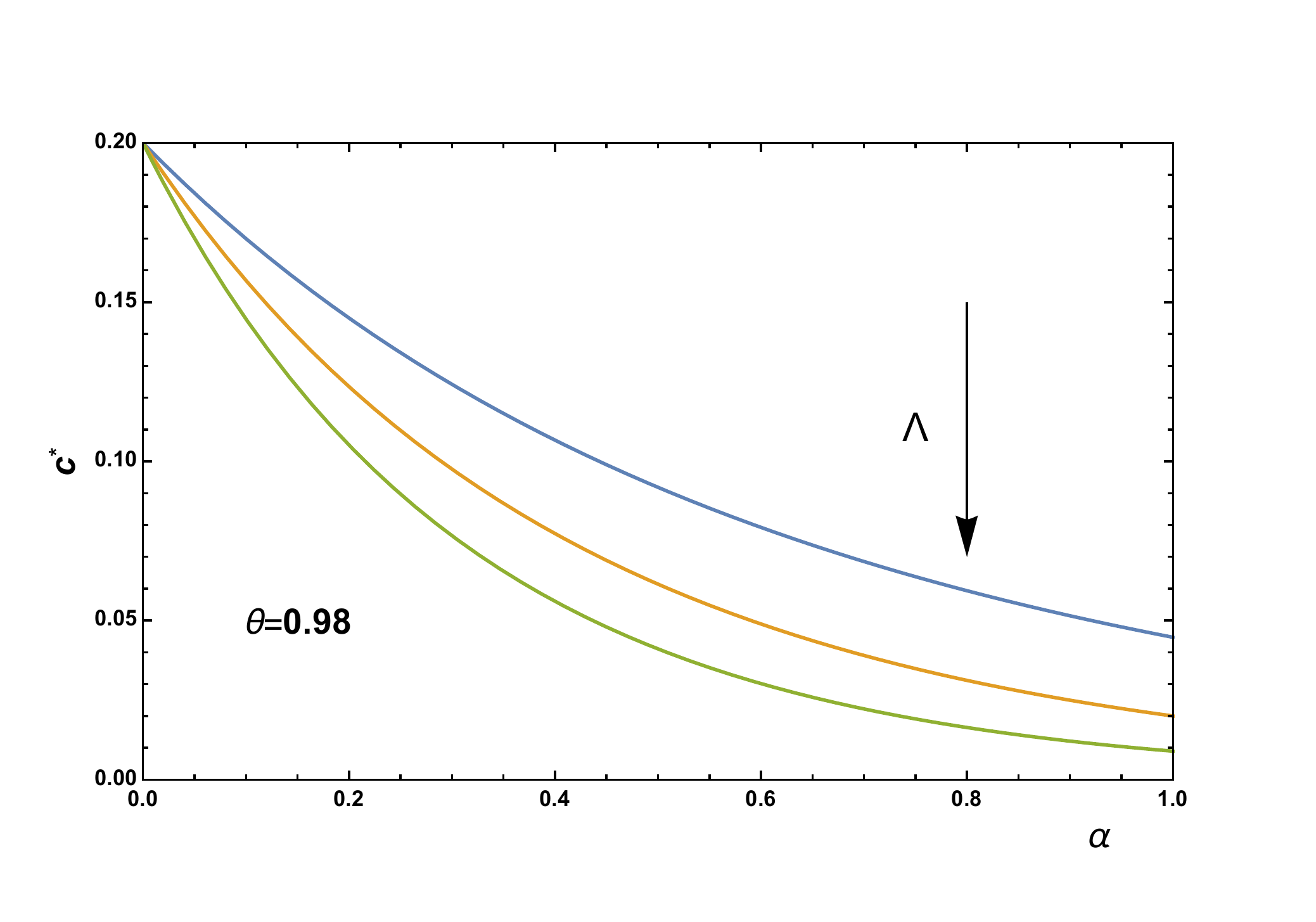}
\end{minipage}\hspace{1.2cm} %
\hspace{.5cm}
\begin{minipage}[c]{0.4\textwidth}%
 \includegraphics[width=3.5in]{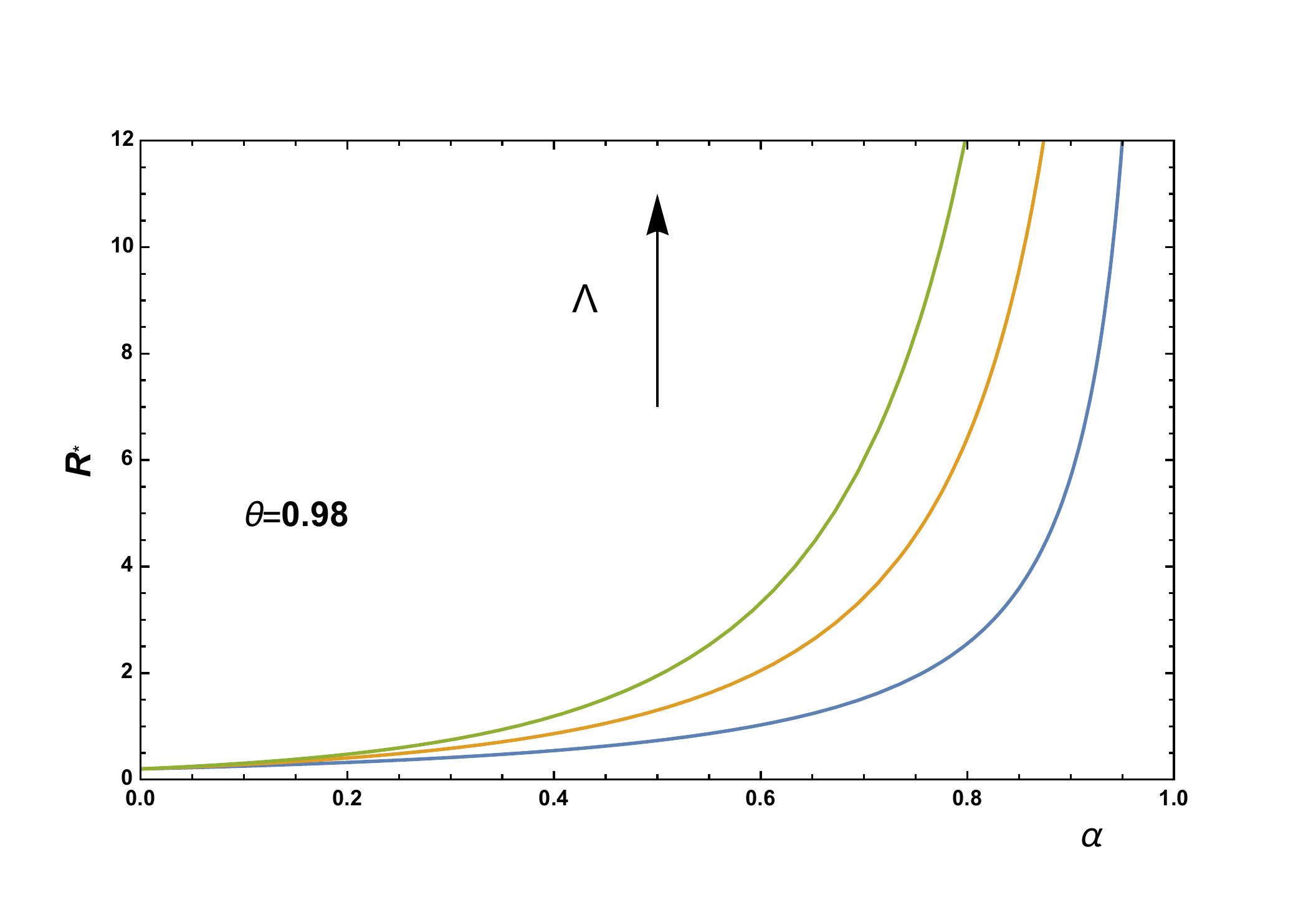} %
\end{minipage}\caption{{Dependency of the velocity of propagation $c^*$ and width of the reaction zone $R^*$ on $\alpha$ for $\Lambda=0.2$ (blue), $\Lambda=1$ (orange) and $\Lambda=5$ (green) for the ignition
 temperature $\theta=0.98$.  The arrow indicates direction of increase of $\Lambda.$}}
\noindent \label{f:t1} 
\end{figure}

 \begin{figure}[h]
\hspace{.5cm}
\begin{minipage}[c]{0.4\textwidth}%
\begin{center}
\includegraphics[width=3.2in]{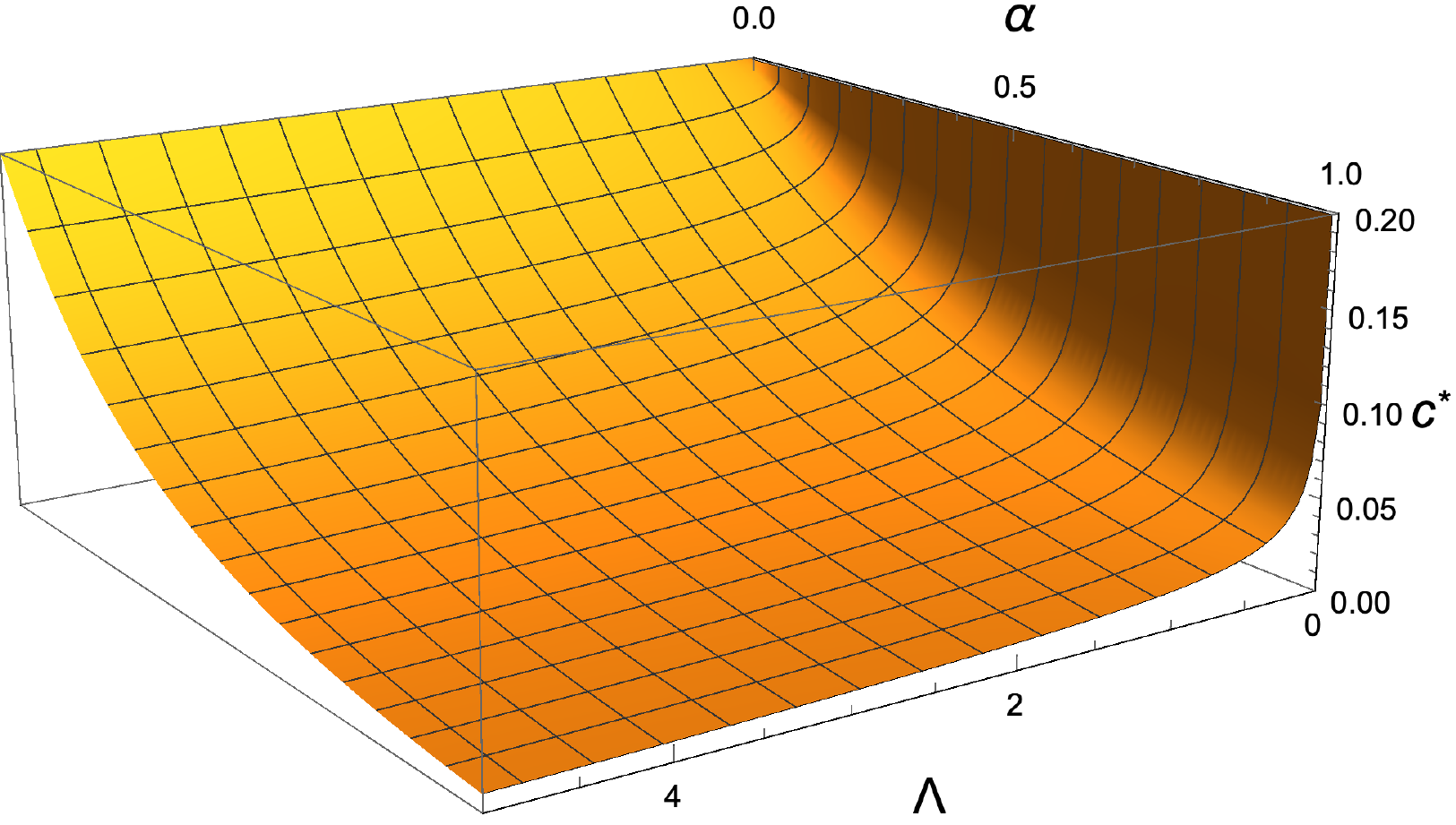}
\par\end{center}%
\end{minipage}\hspace{1.2cm} %
\hspace{1cm}
\begin{minipage}[c]{0.4\textwidth}%
 \includegraphics[width=2.3in]{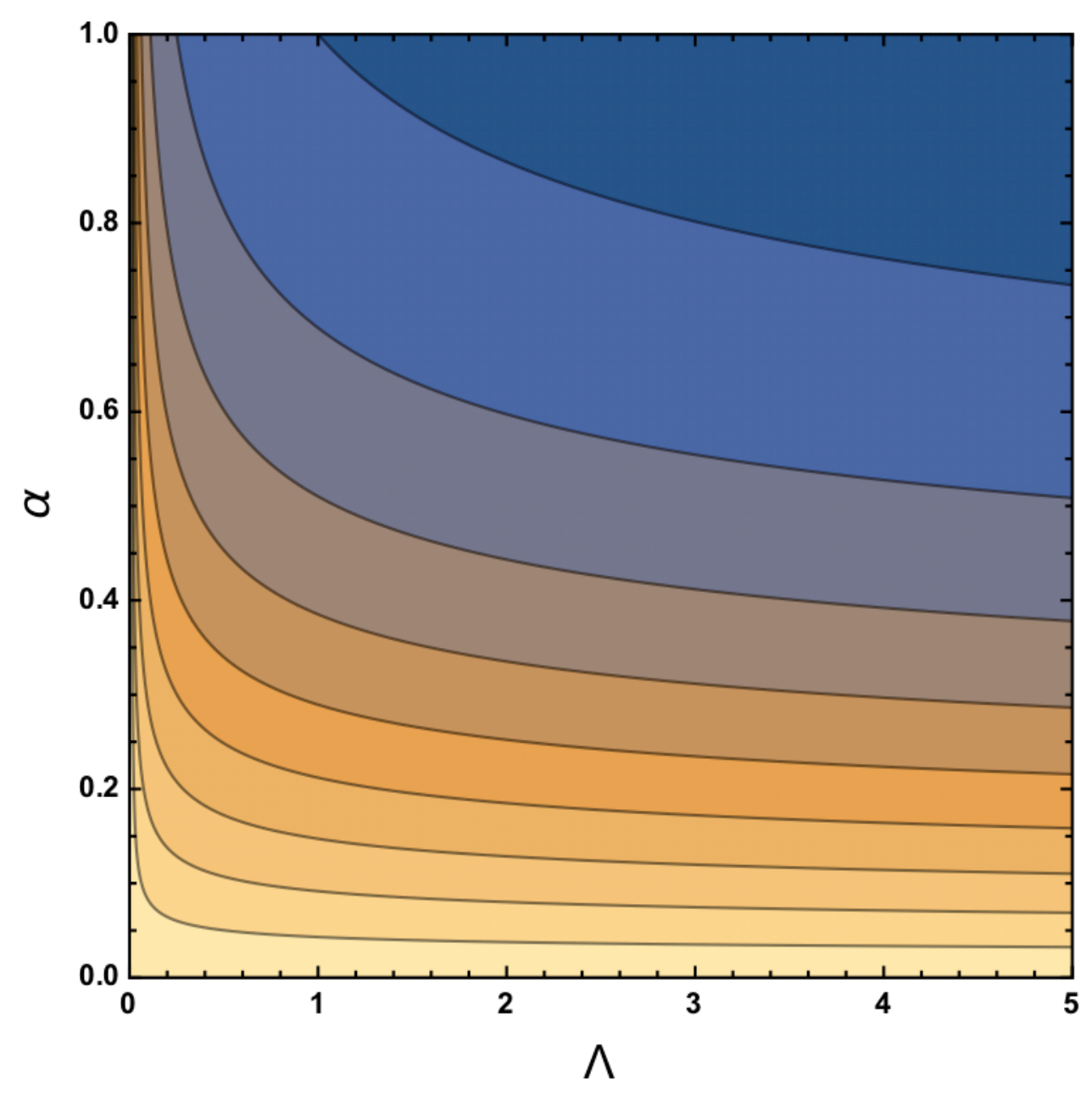} %
\end{minipage}\caption{{ Dependency of the velocity of propagation $c^*$  on $\alpha$ and $\Lambda$ for $\theta=0.98$. Left panel represents
a three dimensional plot of $c^*(\theta=0.98,\Lambda,\alpha)$ and the right plane depicts level sets of this function 
}}
\noindent \label{f:t13d} 
\end{figure}
Moreover, one can verify that  $c^*$ and $R^*$ given by \eqref{eq:as2} fully reproduce the limiting behavior of these functions in the limits  $\alpha\to0$ and $\alpha\to1$.
Indeed, when $\alpha=1,$ the velocity of propagation  and width of the reaction zone are  given by \cite{cnf15}
\begin{eqnarray}
c^*(\theta, \Lambda, \alpha=1)=\left(\left(\frac{\theta}{1-\theta}\right)+\Lambda \left(\frac{\theta}{1-\theta}\right)^2 \right)^{-\frac12}, \quad  R^*(\theta, \Lambda, \alpha=1)=\infty.
\end{eqnarray}
whereas when $\alpha=0,$ we have \cite{cnf15}
\begin{eqnarray}
c^*(\theta, \Lambda, \alpha=0)=R^*(\theta, \Lambda, \alpha=0)=\sqrt{\kappa}.
\end{eqnarray}
where $\kappa$ is defined implicitly as the positive solution of 
\begin{eqnarray}
e^\kappa=\frac{1}{1-\theta \kappa}.
\end{eqnarray}
For $\theta$ near unity these formulas give:
\begin{eqnarray}
c^*(\theta, \Lambda, \alpha=1)\simeq \frac{1-\theta}{\sqrt{\Lambda}},  \quad  c^*(\theta, \Lambda, \alpha=0)\simeq \sqrt{2(1-\theta)} \quad \mbox{for} \quad |\theta-1|\ll1.
\end{eqnarray}
Hence
\begin{eqnarray}
&& c^*(\theta,\Lambda,\alpha\to 1)\to c^*(\theta,\Lambda,\alpha=1), \quad R^*(\theta,\Lambda,\alpha\to 1)\to \infty, \nonumber \\
&&c^*(\theta,\Lambda,\alpha\to 0)\to c^*(\theta,\Lambda,\alpha=0), \quad  R^*(\theta,\Lambda,\alpha\to 0)  \to R^*(\theta,\Lambda,\alpha=0). 
\end{eqnarray}

Now consider the regime of intermediate values of $\theta$, we study this regime numerically.
 Figures \ref{f:t75}, \ref{f:t5}, \ref{f:t25} depict dependency  of the velocity of propagation and width of the reaction zone on
the reaction order $\alpha$ for $\Lambda=0.2,1, 5$ and $\theta=0.75, 0.5, 0.25,$ and Figure \ref{f:amd}  shows dependency of the velocity of propagation on $\alpha$ and $\Lambda$ for $\theta=0.5$.
According to numerics for intermediate values of $\theta,$ the character of the dependency of $c^*$ and $R^*$ on $\alpha$ and $\Lambda$ remains similar to the one for $\theta$ near unity.
However, dependency of $c^*$ on both $\alpha$ and $\Lambda$ becomes weaker as $\theta$ decreases. The dependency of $R^*$ on $\alpha$ in this regime is still very strong, but dependency on $\Lambda$ becomes weaker as
$\theta$ decreases.
 \begin{figure}[h]
\hspace{-.5cm}
\begin{minipage}[c]{0.4\textwidth}%
\includegraphics[width=3.5in]{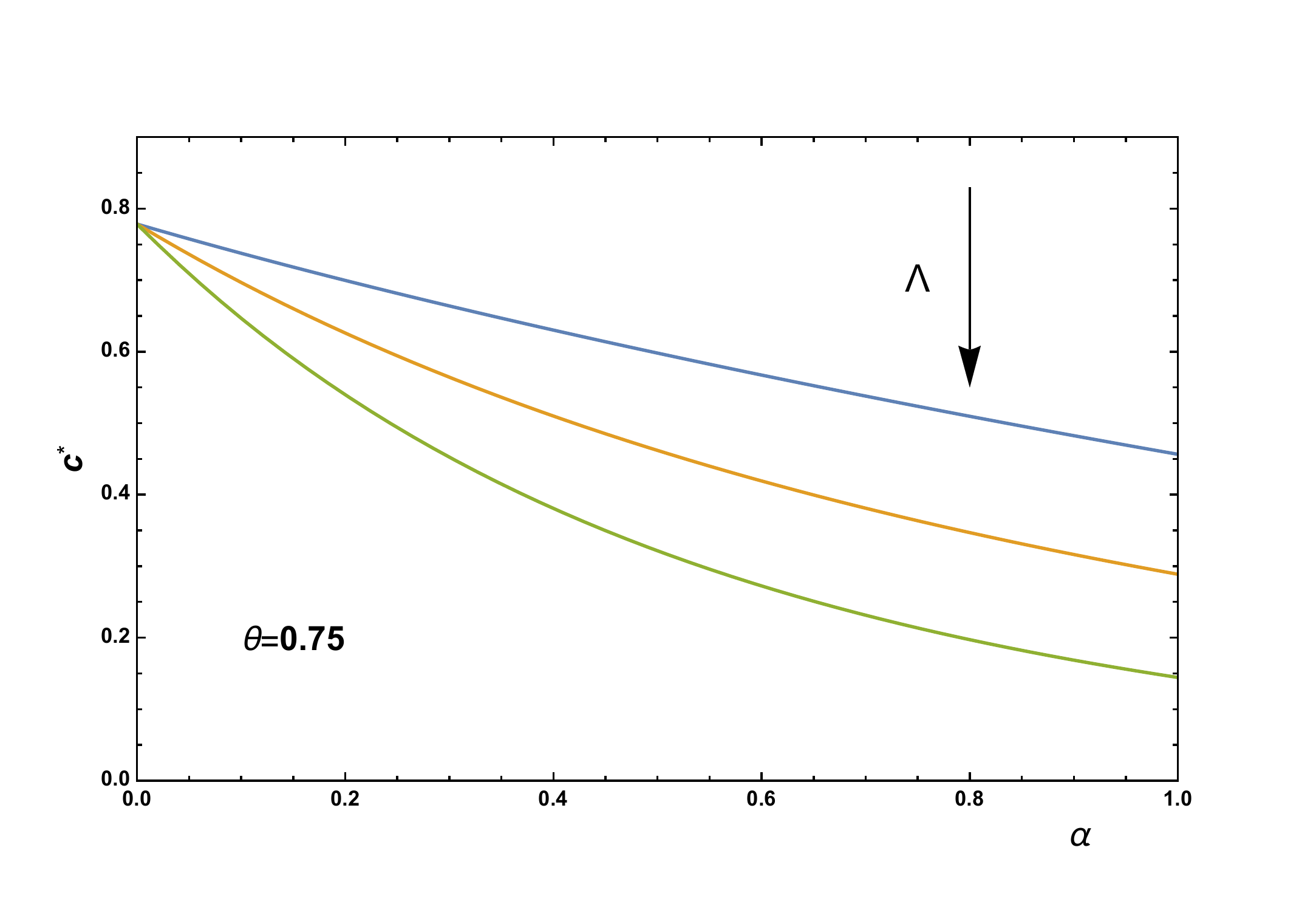}
\end{minipage}\hspace{1.2cm} %
\hspace{.5cm}
\begin{minipage}[c]{0.4\textwidth}%
 \includegraphics[width=3.5in]{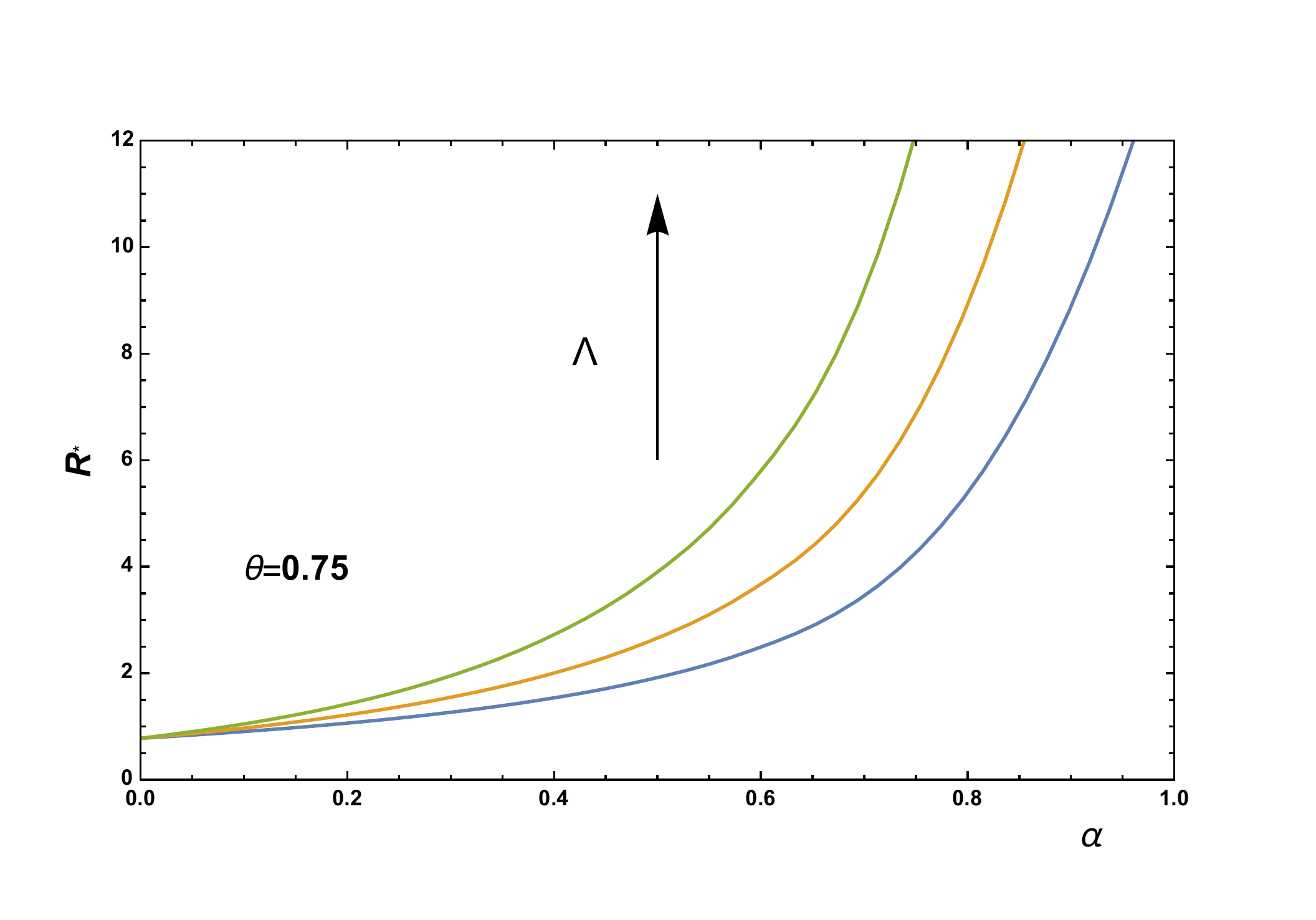} %
\end{minipage}\caption{{Dependency of the velocity of propagation $c^*$ and width of the reaction zone $R^*$ on $\alpha$ for $\Lambda=0.2$ (blue), $\Lambda=1$ (orange) and $\Lambda=5$ (green) for the ignition
 temperature $\theta=0.75$.  The arrow indicates direction of increase of $\Lambda.$}}
\noindent \label{f:t75} 
\end{figure}
\begin{figure}[h!]
\hspace{-.5cm}
\begin{minipage}[c]{0.4\textwidth}%
\includegraphics[width=3.5in]{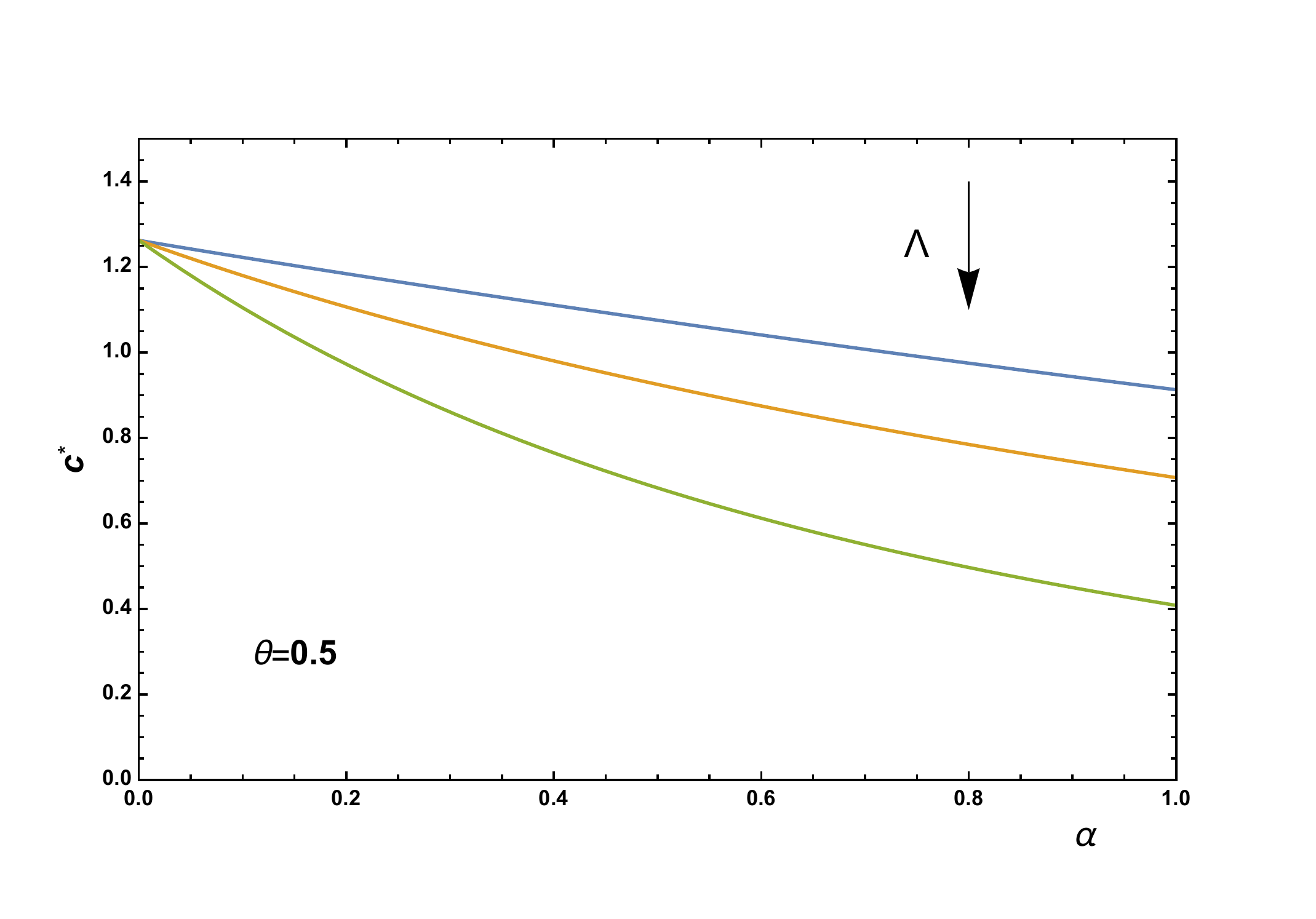}
\end{minipage}\hspace{1.2cm} %
\hspace{.5cm}
\begin{minipage}[c]{0.4\textwidth}%
 \includegraphics[width=3.5in]{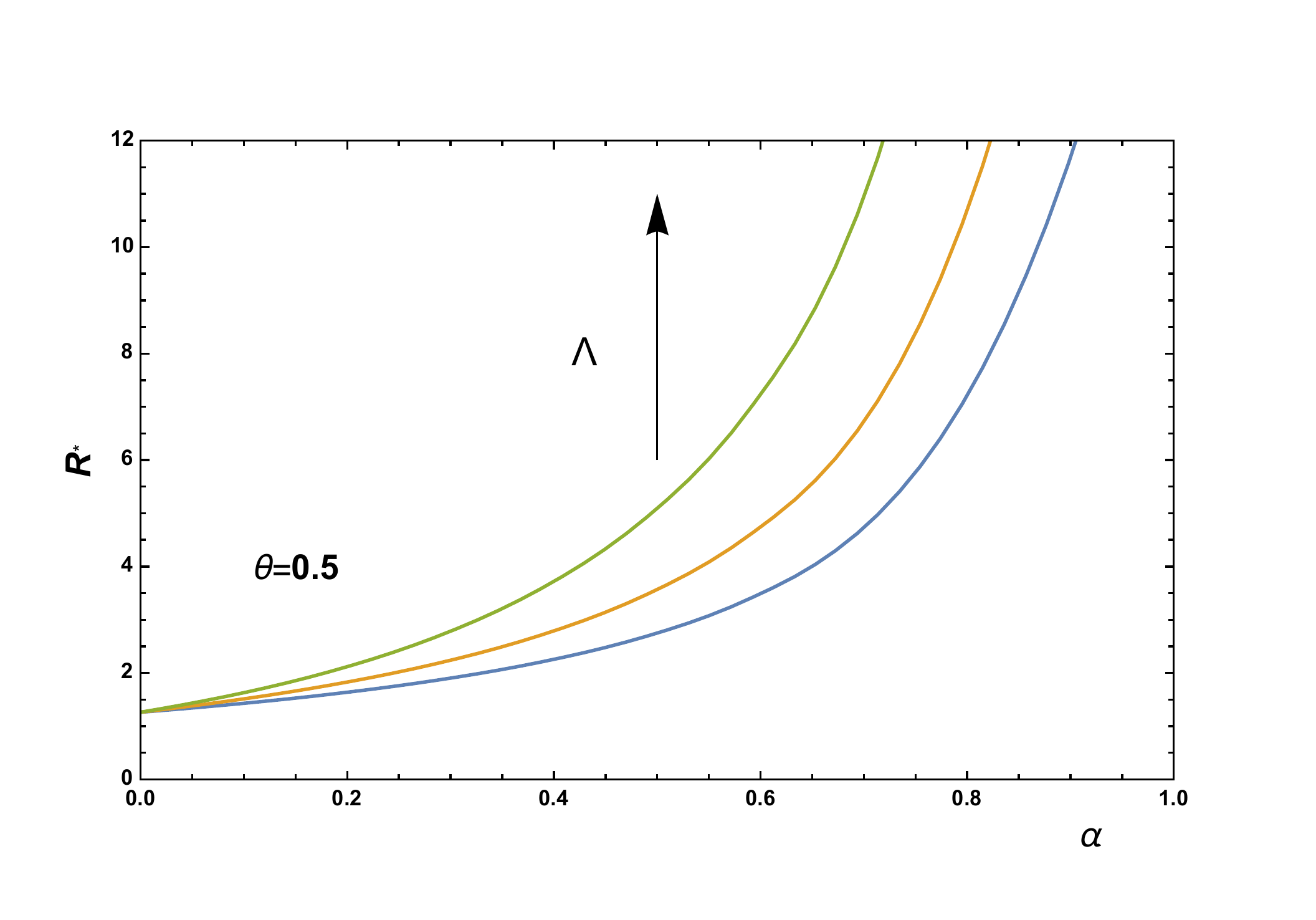} %
\end{minipage}\caption{{Dependency of the velocity of propagation $c^*$ and width of the reaction zone $R^*$ on $\alpha$ for $\Lambda=0.2$ (blue), $\Lambda=1$ (orange) and $\Lambda=5$ (green) for the ignition
 temperature $\theta=0.5$.  The arrow indicates direction of increase of $\Lambda.$}}
\noindent \label{f:t5} 
\end{figure}
\begin{figure}[h!]
\hspace{-.5cm}
\begin{minipage}[c]{0.4\textwidth}%
\includegraphics[width=3.5in]{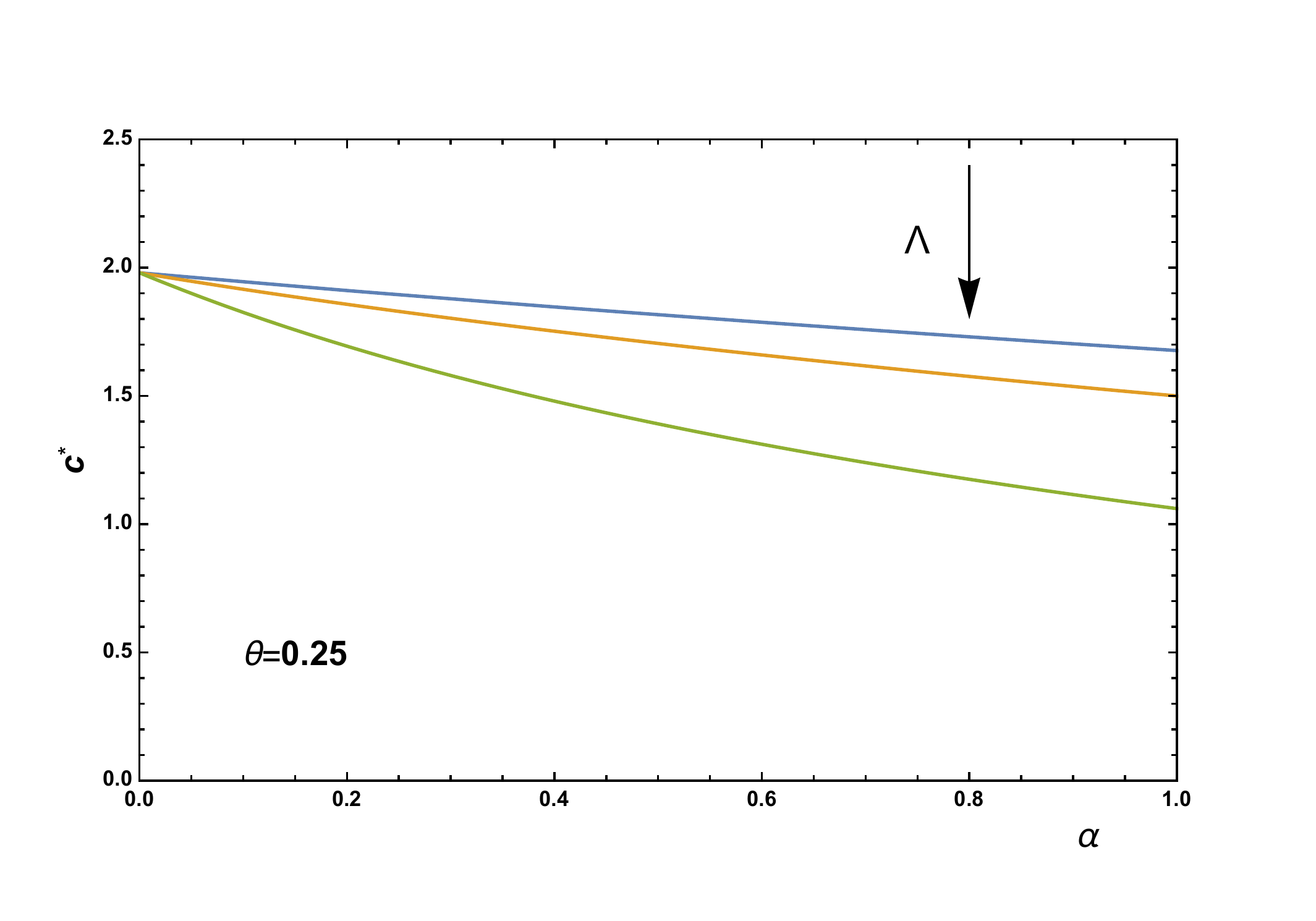}
\end{minipage}\hspace{1.2cm} %
\hspace{.5cm}
\begin{minipage}[c]{0.4\textwidth}%
 \includegraphics[width=3.5in]{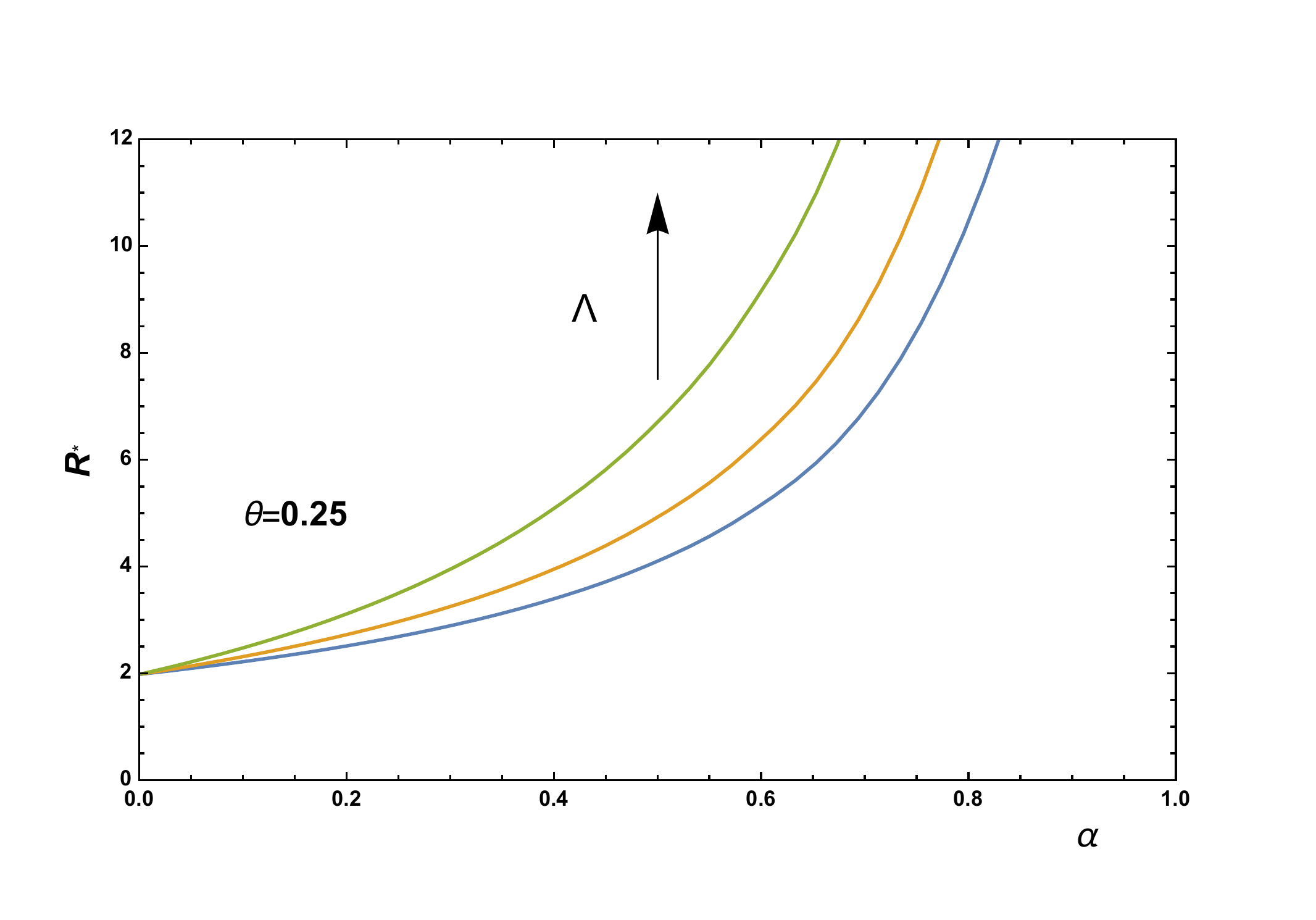} %
\end{minipage}\caption{{Dependency of the velocity of propagation $c^*$ and width of the reaction zone $R^*$ on $\alpha$ for $\Lambda=0.2$ (blue), $\Lambda=1$ (orange) and $\Lambda=5$ (green) for the ignition
 temperature $\theta=0.25$.  The arrow indicates direction of increase of $\Lambda.$}}
\noindent \label{f:t25} 
\end{figure}
\begin{figure}[h]
\hspace{.5cm}
\begin{minipage}[c]{0.4\textwidth}%
\begin{center}
\includegraphics[width=3.2in]{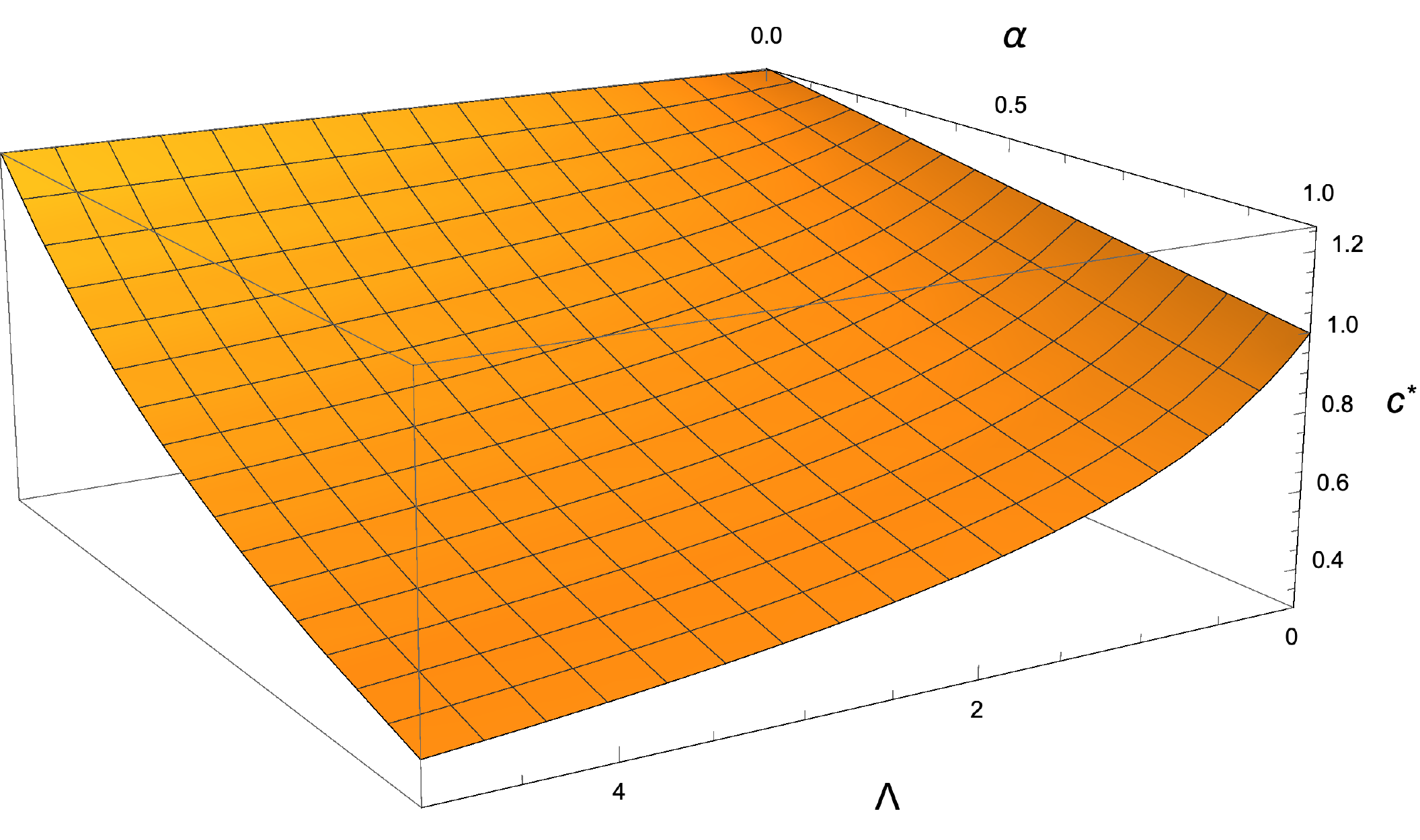}
\par\end{center}%
\end{minipage}\hspace{1.2cm} %
\hspace{1cm}
\begin{minipage}[c]{0.4\textwidth}%
 \includegraphics[width=2.3in]{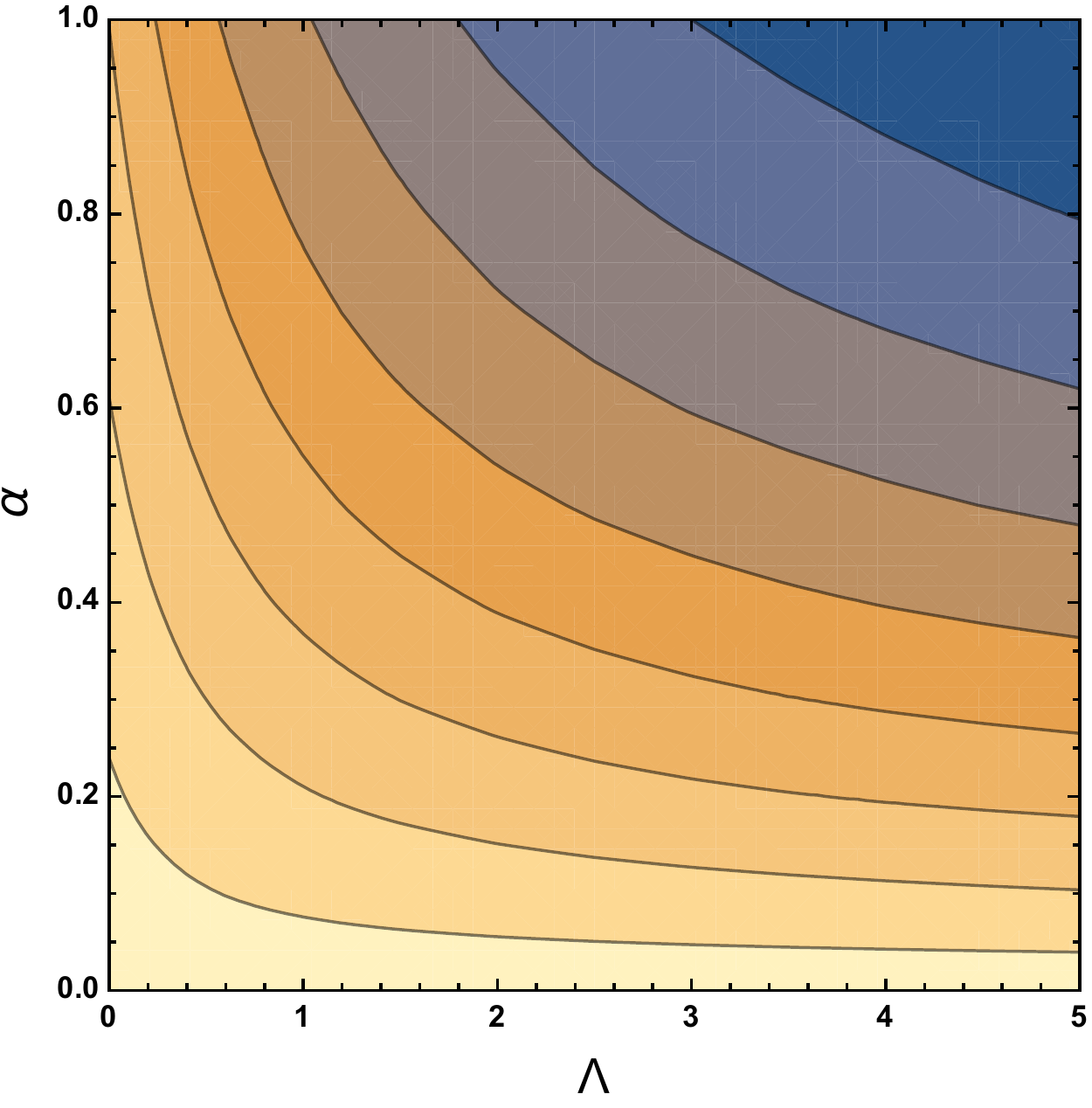} %
\end{minipage}\caption{{ Dependency of the velocity of propagation $c^*$  on $\alpha$ and $\Lambda$ for $\theta=0.5$. Left panel represents
a three dimensional plot of $c^*(\theta=0.5,\Lambda,\alpha)$ and the right plane depicts level sets of this function. 
}}
\noindent \label{f:amd} 
\end{figure}

When $\theta$ becomes sufficiently small, the asymptotic behavior of $\phi(x)$ and $\zeta(x)$ can again be recovered from the asymptotic formula \eqref{eq:21} and reads:
\begin{eqnarray}
\phi(x\vert \Lambda, \alpha)\approx \left(\frac{1}{1-\alpha}\right) \frac{1}{x}, \quad \zeta(x\vert \Lambda, \alpha)= \left(\sqrt{1-\alpha} \right) \sqrt{x} \quad \mbox{for} \quad x\gg 1.
\end{eqnarray}
Therefore, in this regime we have:
\begin{eqnarray}
 c^*(\theta)\approx \frac{1}{\sqrt{\theta}}, \quad R^*(\theta)\approx \frac{1}{ (1-\alpha)} \frac{1}{\sqrt{\theta}}, \quad \theta\ll 1.
\end{eqnarray}
Consequently in this regime,  the velocity of propagation (in the first approximation) depends exclusively on the ignition temperature $\theta,$ and the reaction width is independent of $\Lambda$ and
is an increasing function of $\alpha$.
As in the regime $\theta$ near unity, the velocity of propagation and width of the reaction zone  approaches to formal limits as $\alpha\to 0$ and $\alpha \to 1$.

The discussion above strongly suggests that the velocity of propagation decreases with the increase  of the reaction order.  Consequently, the velocity of propagation with the reaction order $\alpha\in(0,1)$
is bounded from below by the velocity of propagation with the reaction order unity and from above by the velocity of propagation with zero reaction order. This, far from obvious observation, is quite 
in line with the physical intuition as an increase of the reaction order decreases the reaction rate which, in turn, slows the flame front.
Another observation which is less surprising is that the increase of the molecular diffusivity with regard to the thermal diffusivity decreases the speed of the flame front. This is clearly the case for the reactions of first order 
 but remains true for  reaction order $\alpha\in(0,1)$.

We hence formulate the following:
\begin{conjecture} 
The velocity of propagation $c^*(\theta,\Lambda,\alpha)$  is a decreasing  function of all of its arguments whereas the reaction width  $R^*(\theta,\Lambda,\alpha)$ is a decreasing function
of $\theta$ and an increasing function of $\Lambda$ and $\alpha$.
\end{conjecture}

\noindent \textbf{Acknowledgments.}  The work of AM and PVG was supported in a part by US-Israel BSF grant 2020005. PVG would like to thank Fedor Nazarov for multiple valuable discussions and substantial help with proving the main result of this paper.

\section{Appendix}
In this appendix we show that inequality \eqref{eq:a4} is equivalent to \eqref{eq:cl3} and briefly discuss solution of problem \eqref{eq:16} from dynamical systems point of view.

In what follows we denote  the vector field in the right hand side of \eqref{eq:cl2} by ${\bf X}_c$.
The unique critical point of this vector field in $Q$ is the origin $O=(0,0).$ 

System \eqref{eq:cl2} was  thoroughly investigated in \cite[Section 4]{CMB21} using a weak version of Poincar\'e-Bendixson theorem (see \cite{R21}).
The following proposition summarizes the results of Lemmas 4.7 and 4.8 of \cite{CMB21}.

\begin{proposition}\label{p:cl}
There exists a unique  global stable manifold at the origin given by a trajectory of ${\bf X}_c$ in $Q$
converging toward the origin when $t \to 0^-$ such that the orbit defined by this trajectory is the graph of a function ${ p^{\star}}={p}^{\star}({q})$ with  ${q}\in (0,+\infty)$. Moreover, ${ p^{\star}}$ is analytic for ${q}>0$ and extends continuously at $0$ with the value ${p^{\star}}(0)=0$.
\end{proposition}

The trajectory for system \eqref{eq:cl2} can be obtained numerically.  Figures \ref{f:ap1} and \ref{f:ap3} depict such a trajectory on the phase plane  for $\Lambda=1,$  $\alpha=1/2$ in cartesian and polar coordinates respectively. 
Qualitative behavior of the trajectory is similar for other values of parameters $\Lambda,\alpha$.
\begin{figure}[h]
\centering \includegraphics[width=4.in]{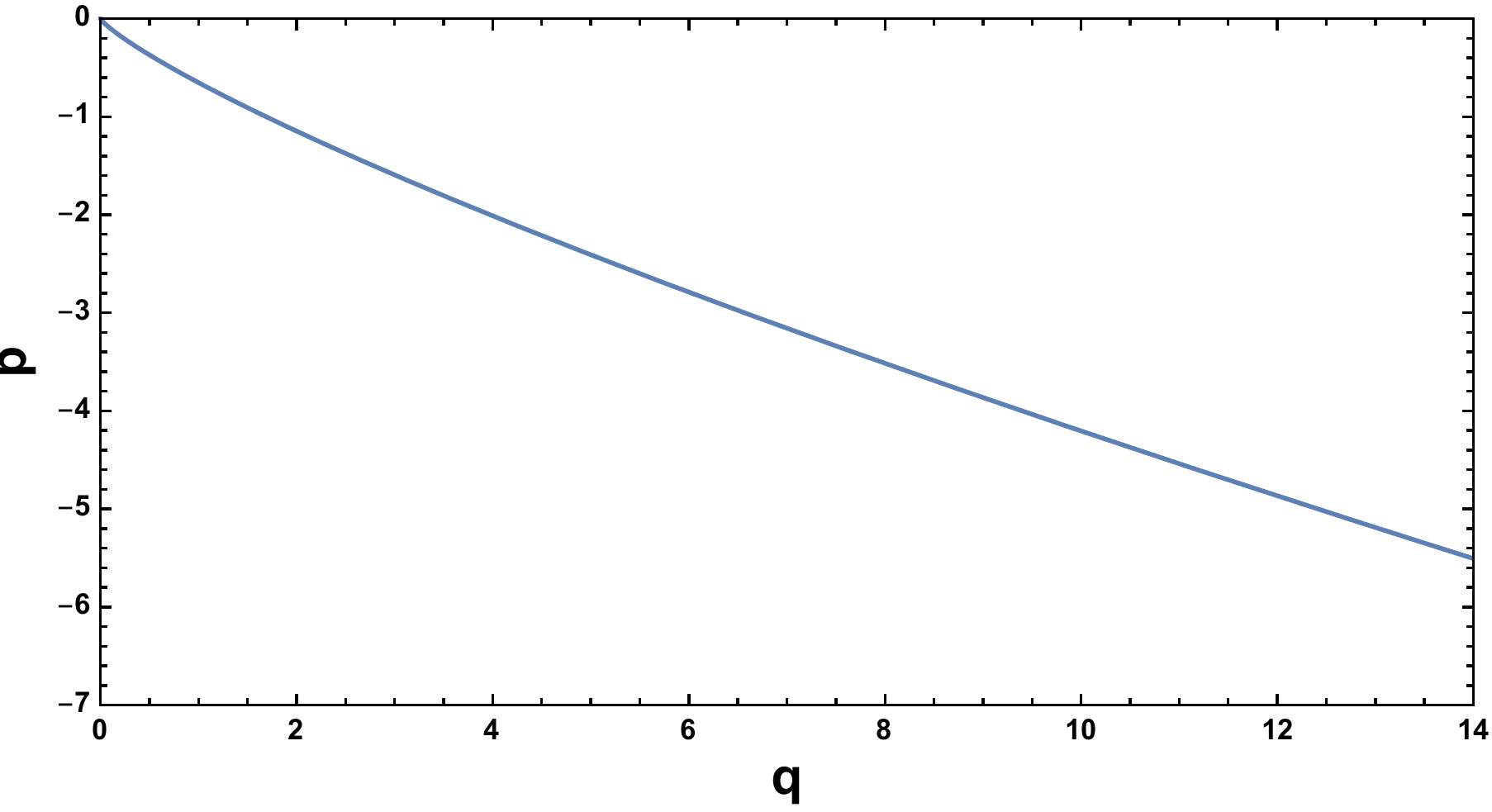}
\caption{Phase portrait for system \eqref{eq:cl2} with $\Lambda=1$ and $\alpha=1/2$ in cartesian coordinates $(q,p).$ }
\label{f:ap1} 
\end{figure}
\begin{figure}[h]
\centering \includegraphics[width=3.in]{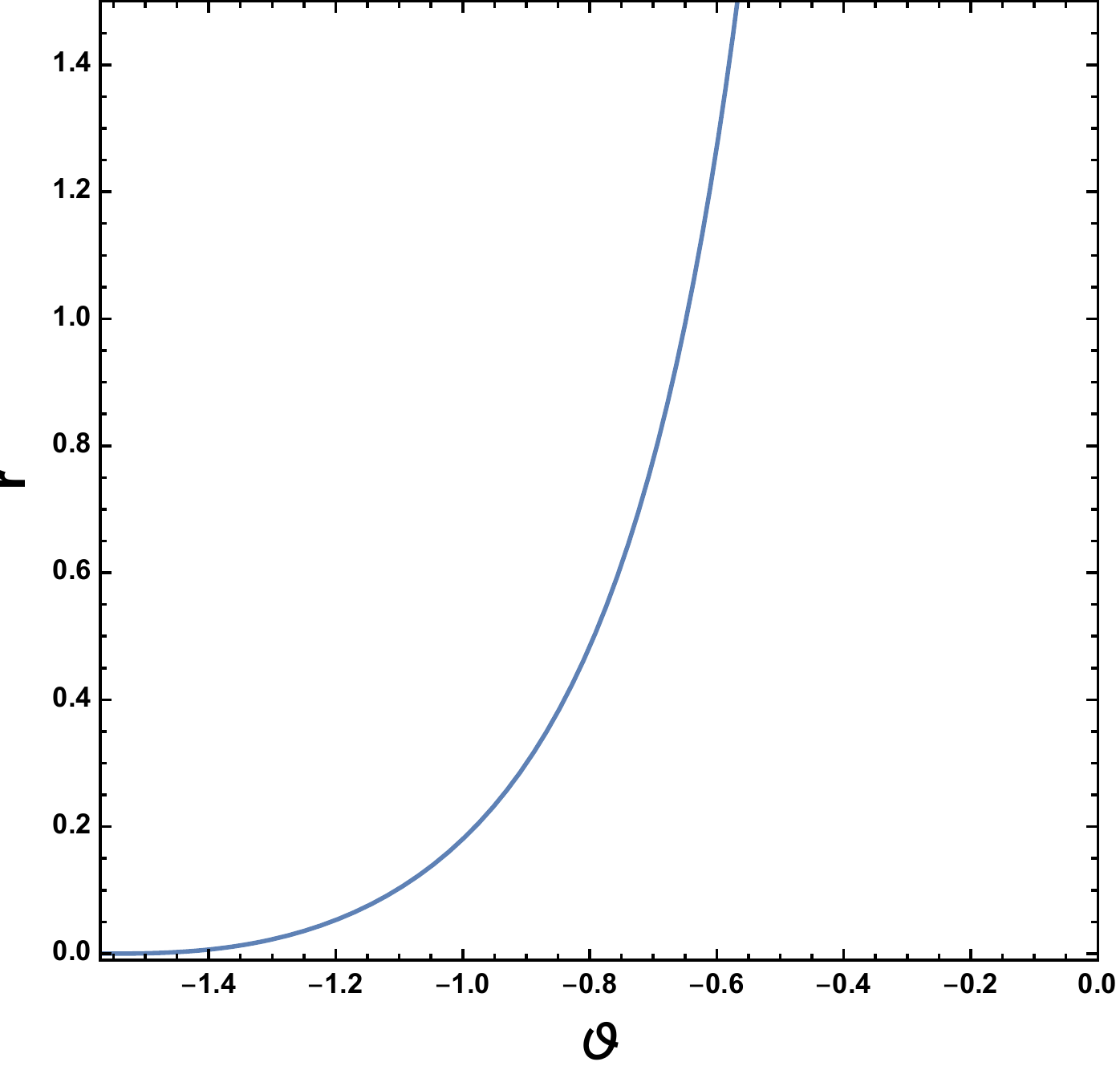}
\caption{Phase portrait for system \eqref{eq:cl2} with $\Lambda=1$ and $\alpha=1/2$ in polar coordinates $(\vartheta,r).$ }
\label{f:ap3} 
\end{figure}

Next observe that
by \eqref{eq:cl3} we have
\begin{eqnarray}
 \displaystyle \tan(\theta(t))=\frac{{p}(t)}{{ q}(t)}.
 \end{eqnarray}
 In view of Proposition \ref{p:cl},  the trajectory $(q,p)(t)$ is smooth and hence (differentiating the expression above)
 we have
\begin{eqnarray} 
\dot{\vartheta} (t)&=&\frac{{ q}(t)\dot{ p}(t)-{ p}(t)\dot{ q}(t)}{r^2(t)}.
\end{eqnarray} 
Since
\begin{eqnarray}
{ q}(t)\dot{ p}(t)-{ p}(t)\dot{ q}(t)=w(x)w''(x)- (w'(x))^2,
\end{eqnarray}
 by \eqref{eq:a4} we obtain \eqref{eq:cl3a}.
 The plot of $\vartheta(t)$ for $\Lambda=1$ and  $\alpha=1/2$ is depicted in Figure \ref{f:ap2}. The function $\vartheta(t)$  is decreasing  on $(-\infty,0)$ and approaches
 $-\pi/2$ as $t\to 0$ and to zero as $t\to -\infty$ regardless of the specific values of parameters $\Lambda,\alpha$ as follows from asymptotic expressions for $w$ (see equation \eqref{eq:21}).
 
\begin{figure}[h]
\centering \includegraphics[width=4.in]{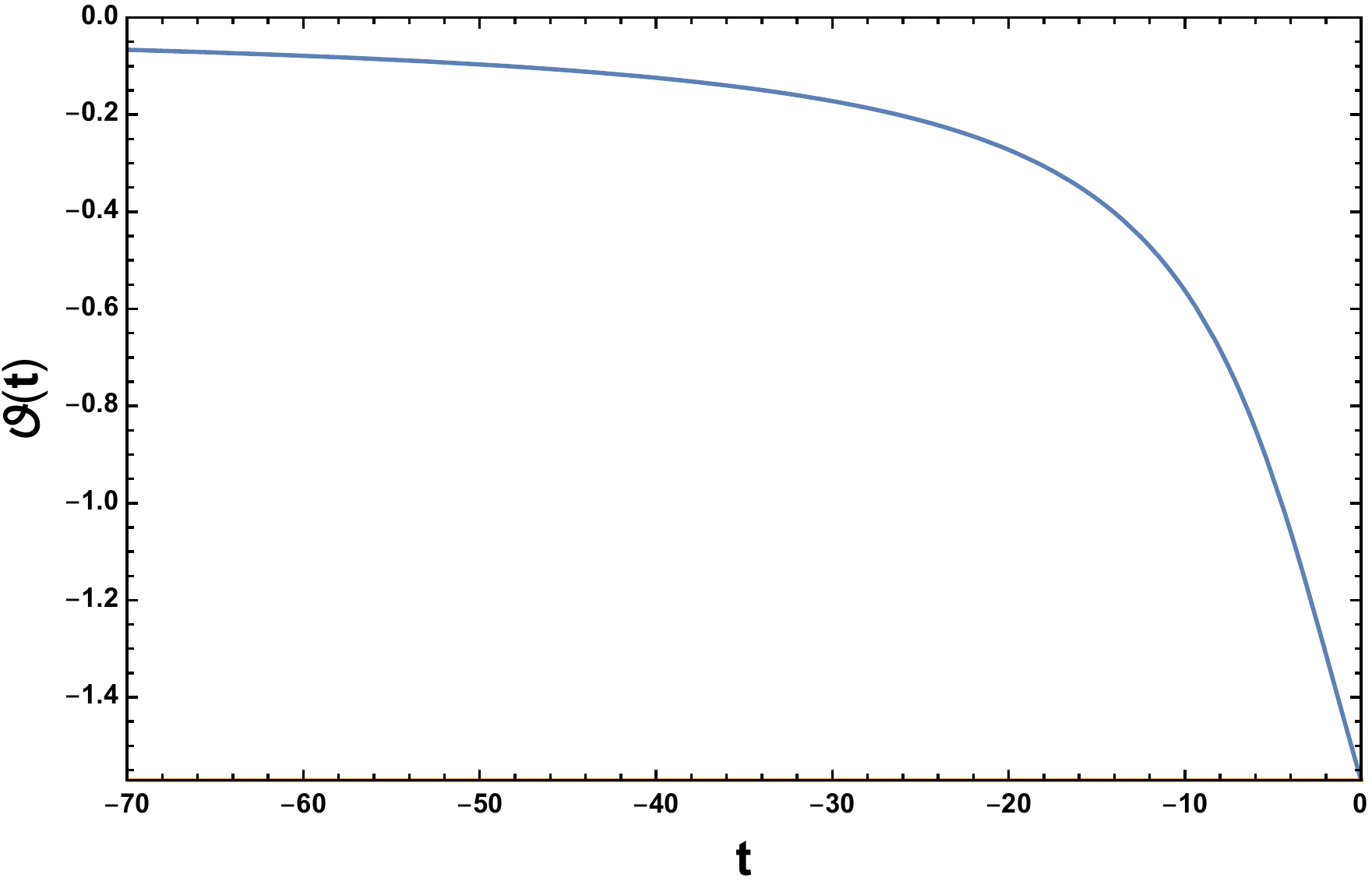}
\caption{Polar angle $\vartheta$ vs. time for the solution of  system \eqref{eq:cl2} with $\Lambda=1$ and $\alpha=1/2.$ }
\label{f:ap2} 
\end{figure}

\bigskip

\noindent {\bf Declaration of Competing Interest:} The authors declare that they have no competing interests.

\bigskip

\noindent {\bf Data availability:}  No data was used for the research described in the article.


\begin{bibdiv} 
\begin{biblist}

\bib{BNS85}{article}{
   author={Berestycki, Henri},
   author={Nicolaenko, Basil},
   author={Scheurer, Bruno},
   title={Traveling wave solutions to combustion models and their singular
   limits},
   journal={SIAM J. Math. Anal.},
   volume={16},
   date={1985},
   number={6},
   pages={1207--1242},
   issn={0036-1410},
   review={\MR{807905}},
   doi={10.1137/0516088},
}


\bib{cnf15}{article}{
author={ Brailovsky, I.},
author={Gordon. P.V.},
author={Kagan, L.},
author={Sivashinsky, G.},
title={Diffusive-thermal instabilities in premixed flames: Stepwise ignition-temperature kinetics},
journal={Combustion and Flame},
volume={162},
date={2015},
pages={2077--2086},
}
		

\bib{CMB21}{article}{
   author={Brauner, Claude-Michel},
   author={Roussarie, Robert},
   author={Shang, Peipei},
   author={Zhang, Linwan},
   title={Existence of a traveling wave solution in a free interface problem
   with fractional order kinetics},
   journal={J. Differential Equations},
   volume={281},
   date={2021},
   pages={105--147},
}



\bib{FK}{book} { AUTHOR = {Frank-Kamenetskii,D. A. }, TITLE
= {Diffusion and heat transfer in chemical kinetics}, PUBLISHER
= { Plenum Press}, ADDRESS = {New York}, YEAR = {1969},

}



\bib{Kan63}{article}{
   author={Kanel', Ja. I.},
title= {A stationary solution to a system of equations in the theory of burning},
journal={Dokl. Akad. Nauk SSSR},
date={1963},
volume={149},
number={2},
pages={367--369},
}





\bib{GK_book}{book}{
   author={Gilding, Brian H.},
   author={Kersner, Robert},
   title={Travelling waves in nonlinear diffusion-convection reaction},
   series={Progress in Nonlinear Differential Equations and their
   Applications},
   volume={60},
   publisher={Birkh\"{a}user Verlag, Basel},
   date={2004},
}


\bib{Law}{book}{
author={ Law, C. K. },
title={Combustion Physics},
publisher={Cambridge University Press},
address={Cambridge},
year={2010}
}


	\bib{R21}{article}{
	author= {Roussarie, R.},
	title={Some Applications of the Poincar\'e-Bendixson Theorem},
	journal={Qual. Theory Dyn. Syst.}
	volume={20},
	number={64}
	date= {2021},
	doi={10.1007/s12346-021-00498-2}
} 


\bib{SW}{article}{
author= {S\'{a}nchez,A. L. },
 author= {Williams,  F.A. },
 title={Recent advances in understanding of flammability characteristics of hydrogen},
 journal={Prog. Energy Combust. Sci.}
 volume={41},
 date= {2014},
 pages={ 1--55}
 } 



\bib{V3}{book}{
   author={Volpert, Aizik I.},
   author={Volpert, Vitaly A.},
   author={Volpert, Vladimir A.},
   title={Traveling wave solutions of parabolic systems},
   series={Translations of Mathematical Monographs},
   volume={140},
   publisher={American Mathematical Society, Providence, RI},
   date={1994},
}

\bib{Volpert}{book}{
   author={Volpert, Vitaly},
   title={Elliptic partial differential equations. Vol. 2},
   series={Monographs in Mathematics},
   volume={104},
   note={Reaction-diffusion equations},
   publisher={Birkh\"{a}user/Springer Basel AG, Basel},
   date={2014},
}
		
		
		

\bib{Sem}{article}{
author={Semenov, N.N.},
title={Thermal theory of combustion and explosion},
Journal={Physics-Uspekhi},
Volume={23},
Issue={3},
year={1940},
pages={251-292},
}

\bib{Will}{book} {
AUTHOR = {Williams, F.},
     TITLE = {Combustion theory},
 PUBLISHER = {Perseus Books},
   ADDRESS = {Reading, MA},
      YEAR = {1985},
    }

\bib{Xin_rev}{article}{
   author={Xin, Jack},
   title={Front propagation in heterogeneous media},
   journal={SIAM Rev.},
   volume={42},
   date={2000},
   number={2},
   pages={161--230},
}


\bib{ZBLM}{book}{
   author={Zel\cprime dovich, Ya. B.},
   author={Barenblatt, G. I.},
   author={Librovich, V. B.},
   author={Makhviladze, G. M.},
   title={The mathematical theory of combustion and explosions},
   publisher={Consultants Bureau [Plenum], New York},
   date={1985},
}



\end{biblist}
\end{bibdiv}
\end{document}